\documentclass{amsart}
\usepackage{latexsym}
\usepackage{amsmath}
\usepackage{amssymb}
\usepackage[all,cmtip,2cell]{xy}
\UseTwocells

\usepackage{bracket}
\usepackage{autobreak}
\usepackage{enumitem}
\usepackage[hypertexnames=false]{hyperref} 
\usepackage[capitalize]{cleveref}
\crefname{equation}{}{}
\crefname{enumi}{}{}

\usepackage[\ifdefined\enabletodonotes\else disable\fi]{todonotes}

\usepackage{autonum} 

\numberwithin{equation}{section}
\newtheorem{thm}{Theorem}[section]
\newtheorem{prop}[thm]{Proposition}
\newtheorem{cor}[thm]{Corollary}
\newtheorem{lem}[thm]{Lemma}
\newtheorem{claim}[thm]{Claim}
\theoremstyle{definition}
\newtheorem{defn}[thm]{Definition}

\theoremstyle{remark}
\newtheorem{rem}[thm]{Remark}

\newcommand{\K}{{\mathbb K}}
\newcommand{\Q}{{\mathbb Q}}
\newcommand{\C}{{\mathcal C}}
\newcommand{\D}{{\mathcal D}}
\newcommand{\e}{\varepsilon}
\newcommand{\R}{{\mathbb R}}
\newcommand{\G}{\mathcal G}
\newcommand{\bul}{\text{\tiny{$\bullet$}}}
\newcommand{\mapright}[1]{%
 \smash{\mathop{%
  \hbox to 1cm{\rightarrowfill}}\limits_{#1}}}
\newcommand{\maprightd}[2]{%
 \smash{\mathop{%
  \hbox to 1.2cm{\rightarrowfill}}\limits^{#1}\limits_{#2}}}
\newcommand{\mapleft}[1]{%
 \smash{\mathop{%
  \hbox to 1cm{\leftarrowfill}}\limits_{#1}}}
\newcommand{\mapleftu}[1]{%
 \smash{\mathop{%
  \hbox to 0.8cm{\leftarrowfill}}\limits^{#1}}}
\newcommand{\maprightu}[1]{%
 \smash{\mathop{%
  \hbox to 1cm{\rightarrowfill}}\limits^{#1}}}
\newcommand{\maprightud}[2]{%
 \smash{\mathop{%
  \hbox to 1cm{\rightarrowfill}}\limits^{#1}_{#2}}}
\newcommand{\mapleftud}[2]{%
 \smash{\mathop{%
  \hbox to 1cm{\leftarrowfill}}\limits^{#1}_{#2}}}

\begin{document}
\title[On multiplicative spectral sequences for nerves]{On multiplicative spectral sequences for nerves and the free loop spaces}

\footnote[0]{{\it 2010 Mathematics Subject Classification}: 55T05, 55T20, 55N25, 55P35, 58A40, 58A10
\\ 
{\it Key words and phrases.} Spectral sequence, Borel construction, free loop space, diffeology.  


Department of Mathematical Sciences, 
Faculty of Science,  
Shinshu University,   
Matsumoto, Nagano 390-8621, Japan   
e-mail:{\tt kuri@math.shinshu-u.ac.jp}
}

\author{Katsuhiko KURIBAYASHI}
\date{}
   
\maketitle

\begin{abstract} We construct a multiplicative spectral sequence converging to the cohomology algebra of the diagonal complex of 
a bisimplicial set with coefficients in a field. 
The construction provides a spectral sequence converging to the cohomology algebra of 
the classifying space of a category internal to the category of topological spaces. By applying the machinery to a Borel construction,  
we explicitly determine the mod $p$ cohomology algebra of the free loop space of the real projective space 
for each odd prime $p$. 
This example is emphasized as an important computational case.
Moreover, we represent generators in the singular de Rham cohomology algebra of the diffeological free loop space of a non-simply connected manifold $M$ with differential forms on the universal cover of $M$ via Chen's iterated integral map. 
\end{abstract}

\section{Introduction}\label{sect:Int}

A {\it stack} is a generalization of a sheaf. More precisely, it is a weak 2-functor from a site to the category of groupoids which satisfies the gluing conditions on objects and morphisms.  
In particular, differentiable 
and topological stacks are obtained by Lie and topological groupoids, respectively, via the stakifications of prestacks associated with such groupoids; see 
\cite[Section 2]{B-X}, \cite{H} and \cite[Chapter 1]{BGNX} for more details. 
In \cite {R-V}, loop stacks are investigated by using diffeological groupoids; 
see \cite{IZ} and Appendix \ref{sect:App} for diffeological spaces. 
It is worthwhile mentioning that the study of {\it orbifolds} is developed with those Lie groupoid presentations; see \cite{A-L-R, M-P}. 

In \cite{B}, Behrend introduced the de Rham cohomology and the singular homology for a differentiable stack, which are those of the classifying space of a Lie groupoid presenting the stack. In particular, the cohomology is invariant under Morita equivalence. 
Thus, such results on stacks and groupoids motivate us to consider how to compute the cohomology of such a classifying space; see Remark \ref{rem:3}.

This manuscript aims to introduce multiplicative spectral sequences computing the cohomology algebras of the classifying spaces of topological categories with coefficients in a field $\K$; see 
Theorems \ref{thm:main} and \ref{thm:App}.  Although there is no application for a stack in this study, 
the product structure demonstrates its power 
in giving more additional structure to the spectral sequence and 
in the computations of the cohomology algebras of non-simply connected spaces. To be more precise, let $LX$ denote the free loop space of a space $X$, $BG$ the classifying space of a finite group $G$ and 
$EG\times_GM$ the Borel construction of a $G$-space $M$. 
Then, Remark \ref{rem:SS_L} enables one to obtain a multiplicative spectral sequence endowed with an $H^*(LBG; \K)$-module structure converging to the cohomology algebra of $L(EG\times_GM)$. 
We also refer the reader to Corollaries
\ref{cor:L_Cotensor} and \ref{cor:L_Cotensor_HH}, Theorem \ref{thm:L_Cotensor} and Proposition \ref{prop:LBG} 
each of which provides a method for computing the Borel cohomology algebra of a $G$-space. 

As a computational example, we explicitly determine the cohomology algebra of the free loop space of the real projective space with coefficients in ${\mathbb Z}/p$ and $\Q$, where $p$ is an odd prime; see Theorem \ref{thm:computation_I}. 
To our knowledge, this result is novel. 
Despite the lack of nontrivial information in the rational cohomology of the even-dimensional 
projective space, the cohomology algebra of the free loop space can be beneficial. 

A {\it diffeological space} is a generalization of a manifold. Therefore, it is crucial to consider smooth (homotopy) invariants of the generalized objects. 
In particular, the de Rham complex and its singular variant of a diffeological space are introduced in \cite{So} and  \cite{K}, respectively. 
In this manuscript, we moreover attempt 
to represent generators in the singular de Rham cohomology of the free loop space of a non-simply connected manifold $M$ 
by using differential forms on the universal cover of $M$ within the framework of {\it diffeology}; see Theorem \ref{thm:DR} and subsequent comments.  As a consequence, 
because of Theorem \ref{thm:computation_I}, we can describe generators of the singular de Rham cohomology algebra of the diffeological 
free loop space of $n$-dimensional real projective space with the volume form on the sphere $S^n$ via Chen's iterated integral map; 
see Theorem \ref{thm:deRhamP}. 
While the original iterated integrals due to Chen work well for the cohomology of the free loop space of a simply connected manifold, the diffeological argument above shows that we can also deal with non-simply connected manifolds in Chen's theory for free loop spaces.  That is an advantage of considering manifolds in diffeology.

An outline of this manuscript is as follows. Section \ref{sect:assertions} introduces a spectral sequence with a multiplicative structure 
for a bisimplicial set. The spectral sequence gives rise to those for topological categories, Borel constructions, and diffeological categories. In Section \ref{sect:proofs}, we establish Theorems \ref{thm:main} and \ref{thm:App}. 
Section \ref{sect:example(s)} describes the computational example mentioned above. 
In Section \ref{GlobalQuotients}, by generalizing the computations in Section \ref{sect:example(s)}, 
we present results of the cohomology algebras of Borel constructions. 
Moreover, we consider spectral sequences for transformation groupoids including an inertia groupoid.  
In Section \ref{sect:deRham}, we investigate the singular de Rham cohomology of the diffeological free loop space of a non-simply connected manifold
by applying results concerning Borel constructions  
in Section \ref{GlobalQuotients}.

Appendix \ref{sect:App2} gives a weak homotopy equivalence between a Borel construction and the free loop space of a quotient space, which is used in the computation in Section \ref{sect:example(s)}. 
In Appendix \ref{sect:App}, we briefly recall the category of diffeological spaces together with adjoint functors between the category of topological spaces. 
Appendix \ref{sect:App3} proves that the diffeological free loop space of smooth maps from $S^1$ is weak homotopy equivalent to the pullback 
of the evaluation map from a path space along the diagonal map in the category of diffeological spaces. 
This result is critical for proving Theorem \ref{thm:DR}. 

\section{A multiplicative spectral sequence for a bisimplicial set and its variants}\label{sect:assertions}
We introduce multiplicative spectral sequences associated with a bisimplicial set by explicitly describing the product structure. 

For a simplicial set $K$, we denote by $C^*(K; \K)$ and  $H^*(K; \K)$ 
the cochain algebra and the cohomology algebra of $K$ with coefficients in a field $\K$, respectively.
Let $\text{Sing}_\bul(X)$ be the singular simplicial set of a space $X$.  
We may write  $H^*_{\text sing}(X; \K)$ or simply $H^*(X; \K)$ for the singular cohomology algebra 
$H^*(\text{Sing}_\bul(X); \K)$.  

Let $S=S_{\bul \bul}$ be a bisimplicial set.  
Then, the vertical face maps give rise to a differential $(d^v)^*$  by the alternating sum 
on the graded algebra 
$C^{p,*}:=C^*(S_{p \bul} ;\K)$ with the usual cup product $\cup$ for any $p\geq 0$. Moreover, we have a double complex 
$\{ \{C^{p, q}\}_{p, q\geq 0}, (d^h)^*, (d^v)^*\}$, where
the differential $(d^h)^*$ is induced by the horizontal face maps of the bisimplicial set $S_{\bul \bul}$. 
A product $\cup_T$ on the total complex $\text{Tot} \, C^{*,*}$ is defined by 
\begin{eqnarray}\label{eq:cupPD}
\omega \cup_T \eta = (-1)^{qp'} (d_{p+1}^h\cdots d_{p+p'}^h)^*\omega \cup (d_0^h \cdots d_{p-1}^h)^*\eta
\end{eqnarray}
for $\omega \in C^{p, q}$ and $\eta \in C^{p', q'}$.
Observe that the differential on $\text{Tot} \,C^{*,*}$ is given by $\delta (\omega) = (d^h)^*(\omega) + (-1)^p(d^v)^*(\omega)$ for $\omega \in C^{p, q}$. 
Thus, we obtain a spectral sequence $\{E_r^{*, *}, d_r\}$ associated with the total complex.

\begin{thm}\label{thm:main} The first quadrant spectral sequence $\{E_r^{*, *}, d_r\}$  with the multiplicative structure defined by (\ref{eq:cupPD}) 
converges to $H^*(\text{\em diag} \,S_{\bul \bul}; \K)$ as an algebra with
\[
E_2^{*,*}\cong H^*(H^*(S_{\bul \bul}, (d^v)^*), (d^h)^*)
\]
as a bigraded algebra, where $\text{\em diag}  \,S_{\bul \bul}$ denotes the diagonal simplicial set of  $S_{\bul \bul}$; see, for example, \cite[Chapter IV, 1]{G-J}. 
Therefore, for a simplicial space $X_\bul = \{X_n, \partial_i, s_j \}$, one has a first quadrant spectral sequence converging to the singular cohomology algebra 
$H^*_{\text sing}(||X_\bul ||; \K)$ with 
$
E_2^{*,*}\cong H^*\big(H^*_{\text sing}(X_\bul, \K), \sum_i (-1)^i \partial_i^*\big)
$
as a bigraded algebra, where $||  \  ||$ denotes the fat geometric realization in the sense of Segal \cite{S}. 
\end{thm}


A prototype of the spectral sequence in Theorem \ref{thm:main} is one stated in \cite[II 6.8. Corollary]{M}; see also \cite[Proposition (5.1)]{S} for a generalized cohomology. The novelty here is that we explicitly provide an algebraic structure in the spectral sequence.

Let $\C=\xymatrix@C15pt@R5pt{[C_1 \ar@<-0.5ex>[r]_-t \ar@<0.5ex>[r]^-s& C_0]}$ be a category internal to $\mathsf{Top}$ the category of topological spaces; that is, structure maps containing the source map $s$ and the target map $t$ are continuous.  
We may drop the maps $s$ and $t$ in the notation of an internal category.  
The nerve functor gives rise to a cosimplicial cohain complex 
\[
n \mapsto C^*(\text{Nerve}_n\C, \K)
\]
and then this induces a cosimplicial abelian group $n \mapsto H^q(\text{Nerve}_n\C, \K)$ for any $q$. 
In what follows, for a cosimplicial abelian group $A^\bullet$, we denote by $H_\Delta(A^\bullet)$ the cohomology of $A^\bullet$. Let $\text{B}\C$ be the classifying space, namely, 
$\text{B}\C =  || \text{Nerve}_\bullet\C ||$, which is the fat geometric realization of the simplicial space 
$\text{Nerve}_\bullet\C$; see \cite{S}. The multiplicative spectral sequence in Theorem \ref{thm:main} is adaptable to many situations. The following results illustrate it with crucial examples.

\begin{thm} \label{thm:App}\text{\em (cf. \cite{B}, \cite[II 6.8.Corollary, IV 4.1.Theorem]{M}, \cite[Corollary 3.10]{Gu-May})} 

\text{\em i)} Let $\C=\xymatrix@C15pt@R5pt{[C_1 \ar@<-0.5ex>[r] \ar@<0.5ex>[r]& C_0]}$ be a category internal to $\mathsf{Top}$. Then there exists a spectral sequence $\{E_r^{*, *}, d_r\}$ converging to 
$H^*(\text{\em B}\C; \K)$ {\it as an algebra} with 
\[
E_2^{p,q}\cong H^p_\Delta(H^q(\text{\em Nerve}_\bul\C; \K)). 
\]

\text{\em ii)} Let $G$ be a topological group and $X$ a $G$-space. Then there exists a spectral sequence converging to the Borel cohomology 
$H^*_G(X; \K):=H^*(EG\times_G X; \K)$ {\it as an algebra} with 
$
E_2^{p,q}\cong H^p(H^q(\text{\em Nerve}_\bul{\mathcal G}; \K)). 
$
Here ${\mathcal G} := \xymatrix@C15pt@R5pt{[G\times X \ar@<-0.5ex>[r] \ar@<0.5ex>[r]& X]}$ denotes the transformation groupoid associated to the $G$-space $X$ whose source map and target map are the projection on the first factor and the action of $G$ on $X$, respectively. 
In particular, one has an isomorphism 
\[
E_2^{p,q}\cong \text{\em Cotor}^{p,q}_{H^*(G)}(\K, H^*(X))
\] 
provided $H^*(G)$ and $H^*(X)$ are locally finite; see 
Remark \ref{rem:SS_L} for a more structure of the spectral sequence in the case where $G$ is a finite group.

\text{\em iii)} Let $\C=\xymatrix@C15pt@R5pt{[C_1 \ar@<-0.5ex>[r] \ar@<0.5ex>[r]& C_0]}$ be a category internal to the category $\mathsf{Diff}$ of diffeological spaces; that is, the sets $C_0$ and $C_1$ of objects and morphisms are diffeological spaces, respectively, and 
structure maps in $\C$ are smooth; see Section \ref{sect:App}.
Let $A_{DR}^*(S^D_\bul(\text{\em Nerve}_\bul (\C)))$ denote the cosimplicial de Rham complex 
of the nerve of $\C$ introduced in \cite[\S 2]{K}; see Remark \ref{rem:3-2} and Appendix \ref{sect:App}. 
Then, there exists a spectral sequence converging to $H^*(\text{\em diag} \,S^D_\bul(\text{\em Nerve}_\bul (\C)); \R)$ {\it as an algebra} with 
\[
E_2^{p,q}\cong H^p_\Delta(H_{DR}^q(\text{\em Nerve}_\bul (\C))). 
\]
Here $H_{DR}^*(\text{\em Nerve}_\bul (\C))$ denotes the cohomology of the complex $A_{DR}^*(S^D_\bul(\text{\em Nerve}_\bul (\C)))$.
\end{thm}

One might expect that the target of the spectral sequence in Theorem \ref{thm:App} iii) is replaced with the cohomology of a more familiar object.   
We discuss the topic in Remarks \ref{rem:EZ}, \ref{rem:3-2} and \ref{rem:3}. 

\begin{rem}\label{rem:4} Let  $\C$ be a category internal to $\mathsf{Top}$. 
Then, by using the polynomial de Rham functor $A_{{PL}}$ (see, for example, \cite[II 10 (a), (b) and (c)]{FHT} and Appendix \ref{sect:S-deRham}) instead of the singular cochain functor in Theorem \ref{thm:App} i), we have a spectral sequence converging to the rational 
cohomology of $\text{B}\C$. 
In this case, an appropriate Sullivan model for each $\text{Nerve}_n\C$ may be useful when computing the $E_2$-term as an algebra; see, for example, 
\cite[Part II]{FHT} for Sullivan models. 
Observe that the product $\wedge_T$ on  $A_{PL}(\text{Nerve}_\bul\C)$ is of the form 
\begin{eqnarray}\label{eq:PDdeRham}
\omega \wedge_T \eta = (-1)^{qp'} \pi_1^*\omega \wedge \pi_2^*\eta
\end{eqnarray}
for $\omega \in A_{{PL}}^q(\text{Nerve}_p\C)$ and 
$\eta \in A_{{PL}}^{q'}(\text{Nerve}_{p'}\C)$, where 
$\pi_1$ and $\pi_2$ are maps assigning $(f_1, \dots, f_p)$ and $(f_{p+1}, \dots, f_{p+p'})$ to 
$(f_1, \dots, f_p, f_{p+1}, \dots, f_{p+p'})$, respectively; see \cite[(6)]{B}. 
\end{rem}

The spectral sequences in Theorem \ref{thm:App} i) and ii) are variants of that in Theorem  \ref{thm:main}. 
Therefore, each of them converges to the cohomology of the total complex $\text{Tot} \,C^{*}(\text{Nerve}_\bul (\C); \K)$ as an algebra. 
The spectral sequence described in Theorem \ref{thm:App} iii) is also constructed by applying Theorem \ref{thm:main}. 
Then, it converges to the cohomology algebra of the total complex $\text{Tot} \, A_{DR}^*(S^D_\bul(\text{Nerve}_\bul (\C)))$ with the same product 
$\wedge_T$ as in (\ref{eq:PDdeRham}). These targets of the convergences are isomorphic to the cohomology algebras  described in Theorem  \ref{thm:App}. This follows from the proof of Theorem \ref{thm:App}. 
Thus, it may be possible to reconstruct the algebra 
structure of the target from that of the $E_\infty$-term in the same way as in \cite[Section 7]{K-Mimura-Nishimoto} with the formula of the product; that is,  we may solve extension problems in the spectral sequences in Theorem \ref{thm:App}. 

Moreover, the formulae (\ref{eq:cupPD}) and (\ref{eq:PDdeRham}) enable us to explicitly consider 
the multiplication on the $E_2$-term of the spectral sequence; see the proof of Proposition \ref{prop:L_0} in which the cohomology algebra of the free loop space of a Borel construction is investigated for low degrees. 

The spectral sequence in \cite[Corollary 3.10]{Gu-May} converging to the homology of the Borel construction $EG\times_GM$ for a $G$-space $M$ has a differential coalgebra structure, and the condition on local finiteness for the homology groups $H_*(G)$ and $H_*(M)$ is {\it not} required in constructing the spectral sequence. Indeed, the torsion product of chain complexes is used in the construction. 
In contrast, in proving Theorem \ref{thm:App} ii), we consider the nerve of a category internal to $\mathsf{Top}$ and apply the K\"unneth theorem. 
Then, the local finiteness for cohomology groups $H^*(G)$ and $H^*(X)$ is required in our theorem. We stress that the multiplicative structure in our spectral sequence is given explicitly by (\ref{eq:cupPD}) without an argument on dualizing the homology.   



\section{Constructions of the spectral sequences}\label{sect:proofs}

The goal of this section is to construct the spectral sequences described in Theorems  \ref{thm:main} and  \ref{thm:App}. 

\begin{proof}[Proof of Theorem \ref{thm:main}] 
We observe that the decreasing filtration $\{F^p\text{Tot} \, C^{*,*}\}_{p\geq 0}$ defined by $(F^p\text{Tot}\, C^{*,*})^n = \bigoplus_{i+j =n, i\geq p}\, C^{i,j}$ provides the spectral sequence. 
Since the product $\cup_T$ in (\ref{eq:cupPD}) preserves the filtration, it induces a multiplicative structure
in the spectral sequence. Thus, for proving the first assertion, it suffices to show that the product in 
$H^*(\text{Tot}\, C^{*,*})$ given by $\cup_T$ is compatible with the cup product on $H^*(\text{diag} \,S_{\bul \bul}; \K)$ under an appropriate isomorphism between the cohomology groups. 
To this end, we use an argument with universal $\delta$-functors; see \cite[Chapter 2]{W}.  

The simplicial identities for the horizontal face maps of 
$S_{\bul \bul}$ enable us to deduce that the cup product $\cup_T$ on $\text{Tot} \,C^{*,*}$ is a cochain map. A direct computation gives the fact. 
We show that a diagram 
\begin{equation}\label{eq:diagram1}
\xymatrix@C15pt@R18pt{
H^*(\text{Tot} \,C^{*,*})\otimes H^*(\text{Tot} \,C^{*,*}) \ar[d]_{H(\text{AW}^*)\otimes H(\text{AW}^*)}^{\cong} 
\ar[r]^-{\cup_T} &  H^*(\text{Tot} \,C^{*,*})  \ar[d]^{H(\text{AW}^*)}_{\cong} \\
\text{dual}\,\pi_*(\text{diag} \, \K(S_{\bul \bul})) \otimes 
\text{dual}\,\pi_*(\text{diag} \, \K(S_{\bul \bul})) \ar[r]_(0.64){\cup} &
\text{dual}\,\pi_*(\text{diag} \, \K(S_{\bul \bul}))
}
\end{equation}
is commutative, where $\cup$ is the usual cup product and $\text{AW}$ denotes the Alexander-Whitney map; see, for example, \cite[8.5.4]{W}.  
The diagram 
\begin{equation}
\xymatrix@C15pt@R18pt{
H_*(\text{dual} \,( \text{Tot} \, CA))\otimes H_*(\text{dual} \, (\text{Tot} \, CA)) \ar[d]_{H(\text{AW}^*)\otimes H(\text{AW}^*)}
\ar[r] &H_*(\text{dual} \,( \text{Tot} \, CA) \otimes \text{dual} \, (\text{Tot} \, CA)) \ar[d]^{H(\text{AW}^*\otimes \text{AW}^*)}\\
\pi_*(\text{dual} \, \text{diag} A) \otimes 
\pi_*(\text{dual} \,\text{diag} A)  \ar[r] & \pi_*(\text{dual} \, \text{diag} A \otimes \text{dual} \,\text{diag} A))
}
\end{equation}
is commutative, 
where the horizontal maps are the canonical ones, $A$ is the bisimplicial vector space $\K(S_{\bul \bul})$ and 
$C$ denotes the double complex functor; see \cite[8.5]{W}.  Then, in order to prove the commutativity of the diagram (\ref{eq:diagram1}), we show that the diagram 
\begin{equation}\label{eq:diagram2}
\xymatrix@C15pt@R18pt{
H_*(\text{dual} \,( \text{Tot} \, CA) \otimes \text{dual} \, (\text{Tot} \, CA)) \ar[d]_{H(\text{AW}^* \otimes \text{AW}^*)}^{\cong} 
\ar[r]^-{\cup_T} &  H_*(\text{dual} \, (\text{Tot} \, CA))  \ar[d]^{H(\text{AW}^*)}_{\cong} \\
 \pi_*(\text{dual} \, \text{diag} A \otimes \text{dual} \,\text{diag} A)) 
 \ar[r]_-{\cup} &
\pi_*(\text{dual}\, \text{diag} \, \K(S_{\bul \bul}))
}
\end{equation}
is commutative. It is proved that $F(\ ):= H_*(\text{dual} \,( \text{Tot} \, C(\ )) \otimes \text{dual} \, (\text{Tot} \, C(\ )))$ and 
$\pi_*(\text{dual} \,\text{diag} (\ ))$ are universal $\delta$-functors from the category of bisimplicial $\K$-vector spaces 
to the opposite category of graded 
$\K$-vector spaces. 
In fact, functors $\text{dual}(\ )$ and  $\text{dual}(\ )\otimes \text{dual}(\ )$ are exact and preserve projectives. Moreover, 
functors $\text{Tot} \, C(\ )$ and $\text{diag} (\ )$ are also exact and preserve projectives; see \cite[8.5.2 and the proof of 8.5.1]{W}. 
Thus, it follows that 
\begin{eqnarray*}
\pi_*(\text{dual} \,\text{diag} (\ ))\!\!\!&=&\!\!\!(L_*\pi_0)\circ(\text{dual} \circ \text{diag}) (\ ) = L_*(\pi_0 \circ \text{dual} \circ \text{diag}) (\ ) \ \ \text{and} 
\end{eqnarray*}
\begin{eqnarray*}
F(\ ) \!\!\!&=&\!\!\! (L_* H_0) (\text{dual} \,( \text{Tot} \, C( \ )) \otimes \text{dual} \, (\text{Tot} \, C(\ ) )) \\
  \!\!\!&=&\!\!\! (L_* H_0)\circ (\text{dual} \otimes \text{dual} \circ \,( \text{Tot} \, C) )( \ )\\ 
  \!\!\!&=&\!\!\!  L_*(H_0\circ \text{dual} \otimes \text{dual} \circ \text{Tot} \, C))( \ ).
\end{eqnarray*}
The result \cite[Theorem 2.4.7]{W} allows us to deduce 
that $F(\ )$ and $\pi_*(\text{dual} \,\text{diag} (\ ))$ are universal $\delta$-functors. 

Since the map $H_0(\text{AW}^*)$ is induced by the identity map on $A_{00}$ and the product $\cup_T$ is nothing but the cup product, it follows that 
the diagram (\ref{eq:diagram2}) is commutative on $H_0$. Therefore, the universality of the functor $F(\ )$ enables us to conclude that  the diagram (\ref{eq:diagram2}) is commutative. Thus, 
the double complex $\text{Tot} \,C^{*,*}$ induces the multiplicative spectral sequence in the assertion. 

In order to prove the latter half of the assertion, we deal with bisimplicial sets and their geometric realizations. 
For a simplicial space $X_\bul$, we see that $| (|\text{Sing}_{\bul'}(X_\bul)|_{\bul'}) | \simeq | \text{diag} \,\text{Sing}_{\bul'} (X_\bul)|$ by the Eilenberg-Zilber theorem; see, for example, \cite[Lemma, page 94]{Q} and \cite[7. Theorem]{G-M}. 

Let $Y$ denote the simplicial space $|\text{Sing}_{\bul'}(X_\bul)|_{\bul'}$ which is the geometric realization with respect to indices $\bul'$. 
Since $Y$ is good in the sense that each degeneracy map is a closed cofibration, it follows that there exists a natural homotopy equivalence $||Y|| \stackrel{\simeq}{\longrightarrow} |Y|$; see \cite[Proposition A.1. (iv)]{S}. 
Moreover,  the counit of the geometric realization functor gives a natural weak homotopy equivalence $Y_n \stackrel{\simeq_w}{\longrightarrow} X_n$ for each $n$. 
Therefore, the equivalence induces a homology isomorphism $||Y|| \to ||X||$; see \cite[Lemma 5.16]{D}. 
It turns out that 
$H^*(\text{diag} \,\text{Sing}(X_\bul), \K) \cong H^*(||X|| , \K)$ as an algebra by a natural map; see \cite[Lemmas 1.2 and 1.3]{Ca} for a homotopical proof of the fact.  
\end{proof} 

\begin{rem}\label{rem:EZ}
One may expect a version of Theorem \ref{thm:main} for a simplicial diffeological space.
The isomorphism in the proof of the latter half of the theorem appears to be well known. To obtain this fact, we take advantage of the Eilenberg--Zilber theorem, which is also applied in \cite{Ca}.  The key to proving the powerful theorem is the use of the homeomorphism $|K \times \Delta[n] | \cong |K| \times |\Delta[n]|$ for a simplicial set $K$ and the standard simplicial set 
$\Delta[n]$. 
However, the diffeological realization functor $| \  |_D$ in the sense of Kihara \cite[Remark 22.1]{Kihara} or Christensen and Wu \cite[Proposition 4.13]{C-W} does {\it not} preserve the product even if $K$ is the standard simplicial set in the example above. Thus, we cannot prove verbatim a diffeological version of Theorem \ref{thm:main}.  
Indeed, it is not easy to replace the target of the spectral sequence in Theorem \ref{thm:App} iii) with a more familiar one for a general category internal to 
$\mathsf{Diff}$. To explain this inconvenience, we have described the proof of the isomorphism. 
\end{rem}

\begin{proof}[Proof of Theorem \ref{thm:App}]
The assertion i) follows from the direct application of Theorem \ref{thm:main}. 
As for the assertion iii), we recall the result \cite[Proposition 3.4]{K} which yields that for any simplicial set $K$, there exists a sequence 
of quasi-isomorphisms of cochain algebras between $C^*(K; \K)$ and $A_{DR}^*(K)$. 
Then, by applying the first half of Theorem \ref{thm:main} to 
the bisimplicial set $S^D_\bullet(\text{Nerve}_\bul(\C))$, we obtain iii). 

To demonstrate the assertion ii), we recall the proof of the result in \cite[Equivariant homology]{B} which describes an equivalence between the Borel cohomology and the cohomology of a groupoid. 

Let ${\mathcal G}$ be the transformation groupoid $\xymatrix@C15pt@R5pt{[G\times X \ar@<-0.5ex>[r] \ar@<0.5ex>[r]& X].}$ The principal $G$-bundle $\pi : EG \times X \to X_G:=EG\times_G X$ defined by the quotient map gives rise to the banal groupoid 
$\widetilde{\mathcal G} := \xymatrix@C15pt@R5pt{[(EG \times X)\times_{X_{G}} (EG \times X) \ar@<-0.5ex>[r]_-{\pi_1} \ar@<0.5ex>[r]^-{\pi_1}& (EG \times X)]}\!,$ where $\pi_i$ is the projection in the $i$th factor. Then, the groupoid $\widetilde{\mathcal G}$ is isomorphic to a transformation groupoid of the form  
$\xymatrix@C15pt@R5pt{[G \times (EG \times X) \ar@<-0.5ex>[r] \ar@<0.5ex>[r]& (EG \times X)].}$ 
In fact, we have a bundle isomorphism 
\[
\xymatrix@C20pt@R10pt{
G\times (EG \times X) \ar[rr]^-\mu_-\cong \ar[rd]_{pr}& & (EG \times X)\times_{X_{G}} (EG \times X) \ar[ld]^-{\pi_2} \\
 & EG\times G &
 }
\]
which is defined by $\mu(g, (e, x)) =((eg^{-1}, gx), (e, x))$. Here $pr$ denotes the projection in the second factor.  

Since the quotient map $\pi$ is a topological submersion, 
it follows from \cite[Lemma 32]{B} that the edge homomorphism $H^*_G(X)=H^*(X_G) \stackrel{\cong}{\to} 
H(\text{Tot}\,C^*(\text{Nerve}_\bul  \widetilde{\mathcal G}))$ is an 
isomorphism of algebras.  The projection $EG\times X \to X$ in the second factor induces a morphism $h : \widetilde{\mathcal G} \to {\mathcal G}$ of 
groupoids. Since $EG$ is contractible, it follows from the spectral sequence argument that $h$ induces an isomorphism 
$h^* : H(\text{Tot}\,C^*(\text{Nerve}_\bul {\mathcal G})) \stackrel{\cong}{\to} H(\text{Tot}\,C^*(\text{Nerve}_\bul  \widetilde{\mathcal G}))$. As a consequence, we have the spectral sequence converging to $H^*_G(X)$ in ii).  The local finiteness of $H^*(G)$ and $H^*(X)$ allows us to conclude that the chain complex $\{ H^*(\text{Nerve}_n \G)\}_n$ with the horizontal differential is nothing but the cobar complex computing the cotorsion functor.  We have the result. 
\end{proof}

As an input datum, our spectral sequence admits the nerve of a category internal to $\mathsf{Diff}$. 
Thus, it is expected that the machinery widely contributes to the calculation of cohomology algebras for not only topological categories but also diffeological ones; see Remark \ref{rem:3-2} below. 

We recall the D-topology functor $D: \mathsf{Diff} \to \mathsf{Top}$ from the category $ \mathsf{Diff}$ of diffeological spaces to that of topological spaces that admits the right adjoint $C$; see Appendix \ref{sect:App}. 
Let ${\mathcal V}_D$ be the class consisting of diffeological spaces $M$ for each of which 
the identity map $id : M \to CDM$ is a weak equivalence in $\mathsf{Diff}$. 
Especially, the result \cite[Theorem 11.2]{Kihara} implies that 
a $C^\infty$-manifold in the sense of \cite[Section 27]{K-M} 
and hence each component of the nerve of a Lie groupoid $\G$ is in ${\mathcal V}_D$. 
In particular,  paracompact manifolds modeled on Hilbert spaces and the space of smooth maps between finite-dimensional manifolds 
are also in  ${\mathcal V}_D$; see \cite[Chapter 11.4]{Kihara} for more details. 

\begin{rem}\label{rem:3-2}
Given a diffeological space $X$, let $S_n^D(X)$ be the set of smooth maps to $X$ from the standard $n$-simplex $\Delta^n_{st}$ endowed with the diffeology in the sense of Kihara \cite[1.2]{Kihara_19}. Then, the cosimplicial set structure on $\Delta^n_{st}$ gives a simplicial set $S_\bullet^D(X):=\{S_n^D(X)\}$.  

It follows from \cite[Proposition 3.2]{Kihara_19} 
that $D(\Delta^n_{st})$ is the standard simplex which is a subspace of $\R^{n+1}$. 
Thus, for a category $\C$ internal to $\mathsf{Diff}$, 
the smoothing theorem \cite[Theorem 1.7]{Kihara} implies that the natural map 
\[
\eta : S^D_\bullet(\text{Nerve}_\bul(\C)) \to \text{Sing}_\bul(D(\text{Nerve}_{\bul}(\C)))
\] 
induced by the functor $D$
is a weak homotopy equivalence of simplicial sets 
provided each $A:=\text{Nerve}_{n}(\C)$ is in the class ${\mathcal V}_D$; see Theorem \ref{thm:smoothing} 
for a particular version of the smoothing theorem. 
Therefore, by \cite[Chapter IV, Proposition 1.7]{G-J}, we see that 
$\text{diag} \, S^D_\bul(\text{Nerve}_\bul(\C)) \simeq_w \text{diag} \! \text{\ Sing}_\bul(D(\text{Nerve}_{\bul}(\C)))$. It turns out that 
the spectral sequence in Theorem \ref{thm:App} iii) converges to the cohomology algebra 
$H^*(|| D(\text{Nerve}_\bullet (\C)) ||; \R)$. 

We apply the argument above to a transformation diffeological groupoid ${\mathcal G} := \xymatrix@C15pt@R5pt{[G\times N \ar@<-0.5ex>[r] \ar@<0.5ex>[r]& N]}$ for which $G$ is a Lie group and $N$ is in ${\mathcal V}_D$.
 It follows from \cite[Lemma 4.1]{C-S-W} that the natural map $D(X\times Y) \to D(X)\times D(Y)$ is a homeomorphism if 
$D(X)$ is locally compact Hausdorff. This implies that $D(G^{\times n} \times N) \cong G^{\times n}\times D(N)$. 
Moreover, the functor $C$ is the right adjoint to $D$ and hence $C$ preserves the products. We conclude that each component 
$\text{Nerve}_n(\G)= G^{\times n} \times N$ of the nerve of $\G$ is in  ${\mathcal V}_D$.  As a consequence, 
the spectral sequence in Theorem \ref{thm:App} iii) converges to 
$H^*(|| D(\text{Nerve}_\bullet (\G)) ||; \R)\cong H^*(|| \text{Nerve}_\bullet (D\G) ||; \R)\cong H^*(EG\times_G DN; \R)$
as algebras, where $D\G$ denotes a topological groupoid of the form 
$\xymatrix@C15pt@R5pt{[G\times DN \ar@<-0.5ex>[r] \ar@<0.5ex>[r]& DN]}$; see the proof of Theorem \ref{thm:App} 
for the second isomorphism. 
Thus, we may describe the generators of $H^*(EG\times_G D(N); \R)$ with differential forms on $N$, although we do not pursue such a topic in this article; see Section \ref{sect:deRham} for a related topic. 
%
\end{rem}


\begin{rem}\label{rem:3}
Let $\G$ be a Lie groupoid presenting a differentiable stack ${\mathfrak X}$. By definition, 
the de Rham cohomology $H_{DR}^*({\mathfrak X})$ of the stack is the cohomology of the total complex of the bigraded de Rham complex 
$A:=\wedge^*(\text{Nerve}_{\bul} (\G))$; see \cite[Definition 9]{B}. 
By Proposition \ref{prop:Afactor_map}, the {\it factor map} 
$\alpha : \wedge^*(\text{Nerve}_{\bul} (\G)) \to A_{DR}^*(S^D_\bul(\text{Nerve}_\bul (\G))$
is a natural quasi-isomorphism. 

Moreover,  
the result \cite[Proposition 3.4]{K} asserts that there exists a sequence 
of natural quasi-isomorphisms of cochain algebras between $C:=C^*(S^D_\bul(\text{Nerve}_\bul (\G); \K)$ and $A_{DR}^*(S^D_\bul(\text{Nerve}_\bul (\G))$. 
Thus, we have a sequence of quasi-isomorphisms between the total complexes $(\text{Tot}\, C,  \cup_T)$ and $(\text{Tot}\, A, \wedge_T)$ which preserve products; see  (\ref{eq:cupPD}) and (\ref{eq:PDdeRham}) for the formulae of $\cup_T$ and $\wedge_T$, respectively. Therefore, because $\text{Nerve}_{\bul} (\G)$ is a manifold, the proof of the latter half of Theorem \ref{thm:main} and the same argument as in Remark \ref{rem:3-2} enable us to conclude that $H_{DR}^*({\mathfrak X}) \cong  H^*(B{\mathcal G}, {\mathbb R})$ {\it as an algebra}. 
We observe that the product $\wedge_T$ coincides with that in the total complex mentioned in \cite[(6)]{B}. 
\end{rem}

\section{Computational examples}\label{sect:example(s)}

The aim of this section is to compute the cohomology algebra of the free loop space of the real projective space.  
Let $G$ be a {\it discrete} group acting on a topological space $M$. For $g \in G$, we define 
$
{\mathcal P}_g(M) :=\{ \gamma : [0, 1] \to M  \mid \gamma(1) = g\gamma(0) \}
$
which is a subspace of the space of continuous paths $M^{[0,1]}$ from the interval $[0, 1]$ to $M$ with the compact-open topology. Moreover, we put 
\begin{eqnarray}\label{eq:P_G}
{\mathcal P}_G(M) := \coprod_{g\in G}({\mathcal P}_g(M) \times \{g\}).
\end{eqnarray}
Then, the space ${\mathcal P}_G(M)$ admits a $G$-action defined by $h\cdot (\gamma, g) = ({}_h\gamma, hgh^{-1})$, where ${}_h\gamma(t) = h\cdot \gamma(t)$. For a space $X$, let $LX$ denote the free loop space of $X$, namely, the space of continuous maps from the circle $S^1$ to $X$.  
Let $G \to M \stackrel{p}{\to} M/G$ be a principal $G$-bundle. Proposition \ref{prop:An_orbifold_stack} 
enables us to obtain 
a weak homotopy equivalence 
\begin{equation}\label{eq:w}
\overline{p}: EG \times_G{\mathcal P}_G(M) \stackrel{\simeq_w}{\longrightarrow} L(M/G) 
\end{equation} 
which is induced by the projection $p : M \to M/G$. 
Thus, we see that the spectral sequence in Theorem \ref{thm:App} ii) computes the cohomology algebra of a transformation groupoid of the form 
$\G =\xymatrix@C15pt@R5pt{[G\times {\mathcal P}_G(M)  \ar@<-0.5ex>[r] \ar@<0.5ex>[r]& {\mathcal P}_G(M) ]}$ and hence that of 
$L(M/G)$.

Suppose that $M$ is simply connected. 
For each $g \in G$, to compute $H^*({\mathcal P}_g(M); \K)$ with coefficients in a field $\K$, 
we may use the Eilenberg--Moore spectral sequence (henceforth EMSS for short) for the pullback diagram 
\begin{eqnarray}\label{eq:PB_diagram}
\xymatrix@C25pt@R12pt{
{\mathcal P}_g(M) \ar[r] \ar[d] & M^{[0, 1]} \ar[d]^{\e_0\times \e_1} \\
M \ar[r]_-{1\times g} & M\times M, 
}
\end{eqnarray}
where $\e_i$ is the 
evaluation map at $i$ for $i = 0, 1$ and $g$ denotes the map induced by the action on $M$ with the element $g$. 
We observe that the EMSS $\{E_r^{*,*}, d_r\}$ 
converges 
to $H^*({\mathcal P}_g(M); \K)$ as an {\it algebra} with 
\[
E_2^{*,*} \cong \text{Tor}_{H^*(M; \K)\otimes H^*(M; \K)}^{*,*}(H^*(M; \K)_g, H^*(M; \K))
\]
as a {\it bigraded algebra}. Here 
$H^*(M; \K)_g$ denotes the cohomology algebra $H^*(M; \K)$ endowed with the right $H^*(M; \K)\otimes H^*(M; \K)$-action defined by 
$a\cdot(\lambda\otimes\lambda') = a(\lambda g^*(\lambda'))$ for $a\in  H^*(M; \K)_g$ and $\lambda, \lambda' \in H^*(M; \K)$. 

For an element $h$, the $G$-action on ${\mathcal P}_G(M)$ induces the map $h_* :   {\mathcal P}_g(M) \to  {\mathcal P}_{hgh^{-1}}(M)$ which fits in the commutative diagram 
\begin{eqnarray}\label{eq:actions}
\xymatrix@C10pt@R5pt{
    {\mathcal P}_g(M)\ar[rd]_(0.45){h_*} \ar[rr] \ar[dd] & & M^{[0, 1]} \ar[rd]^{h_*} \ar[dd]_(0.3){\e_0\times \e_1} | {\hole} & \\
   &  {\mathcal P}_{hgh^{-1}}(M) \ar[rr]  \ar[dd] & &  M^{[0, 1]} \ar[dd]^-{\e_0\times \e_1}  \\
M \ar[rd]_(0.45){h} \ar[rr]^(0.3){1\times g} | {\hole}  & &M\times M  \ar[rd]^{h\times h} &\\
   & M \ar[rr]_-{1\times hgh^{-1}} & & M\times M.
}
\end{eqnarray}
Then, the naturality of the EMSS gives rise to a morphism of spectral sequences that is compatible with the map 
$(h_*)^* : H^*( {\mathcal P}_{hgh^{-1}}(M); \K) \to H^*( {\mathcal P}_g(M); \K)$. 

In the rest of this section, 
by utilizing the spectral sequence in Theorem \ref{thm:App} ii), we determine the cohomology algebra of the free loop space $L{\mathbb R}P^n$ of the real projective space with coefficients in ${\mathbb Z}/p$  for $p \neq 2$. Moreover, we investigate the mod $2$ cohomology algebra of $L{\mathbb R}P^n$ in a more general context; see Proposition \ref{prop:L_0}.  

A portion of the computations, in general,  follows from a description of the cotorsion functor with the cohomology of a group; see Lemma \ref{lem:Cotor_G}, Theorem \ref{thm:L_Cotensor} and Corollary \ref{cor:L_Cotensor} below. However, 
we here compute the functor with the cobar complex to see what is happening in the $E_1$-term of the spectral sequence that we apply in the computation. 
\begin{thm}\label{thm:computation_I}
Let $p$ be an odd prime or $0$ and $m$ a positive integer, then as algebras, 
\begin{eqnarray*}
H^*(L\R P^{2m+1}; {\mathbb Z}/p) &\cong& H^*(LS^{2m+1}; {\mathbb Z}/p)\oplus H^*(LS^{2m+1}; {\mathbb Z}/p)  \\
  &\cong& (\wedge (y)\otimes \Gamma[\overline{y}])^{\oplus 2}    \  \  \      \text{and} \\
 H^*(L\R P^{2m}; {\mathbb Z}/p) &\cong& (\wedge(x\otimes u) \otimes \Gamma[w])\oplus \mathbb{Z}/p, 
\end{eqnarray*}
where $\deg y = 2m+1$, $\deg \overline{y} = 2m$, $\deg (x\otimes u) = 4m-1$, $\deg w = 4m-2$ and ${\mathbb Z}/0 := {\mathbb Q}$. 
\end{thm}


Before starting the proof of Theorem \ref{thm:computation_I}, we recall a result on a right $G$ action on a vector space and the right $G$-coaction associated with the action. Let $G$ be a finite group and $\eta : \K[G]\otimes V \to V$ a 
$G$-action on a finite-dimensional vector space $V$. 
The left $G$-action gives rise to the right $G$-action $\varphi : V^\vee \otimes\K[G] \to V^\vee$ defined by $\varphi(f\otimes g)(v) = f(\eta(g\otimes v))$ for 
$v\in V$. Moreover, we see that 
the adjoint $ad(\varphi) : V^\vee \to \K[G]^\vee \otimes V^\vee$ coincides with the dual coaction $\eta^\vee : V^\vee \to \K[G]^\vee \otimes V^\vee$. 
We use the fact in the computation below without mentioning that. 

In what follows, we may omit the coefficients in the cohomology groups that we deal with. 

\begin{proof}[Proof of Theorem \ref{thm:computation_I}] 
The antipodal action of $G:={\mathbb Z}/2$ on the sphere $S^n$ gives rise to the real projective space 
$\R P^n = S^n/G$. 
By applying Theorem \ref{thm:App} ii) to the groupoid 
$\xymatrix@C15pt@R5pt{[G\times {\mathcal P}_G(S^n)  \ar@<-0.5ex>[r] \ar@<0.5ex>[r]& {\mathcal P}_G(S^n) ]}$\!, we have a spectral sequence 
$\{E_r^{*,*}, d_r\}$ converging to $H^*(L\R P^n; {\mathbb Z}/p)$ with 
$E_2^{*,*}\cong \text{Cotor}^{*,*}_{H^*(G)}({\mathbb Z}/p, H^*({\mathcal P}_G(S^n)))$ as an algebra. Since $G$ is abelian, it follows that 
the $G$ action on ${\mathcal P}_G(S^n)$ is restricted to each ${\mathcal P}_g(S^n)$ for $g \in G$. 
Then, we see that $L({\mathbb R}P^n)\simeq_w EG\times_G{\mathcal P}_G(S^n) = \coprod_{g \in G}\big(EG \times_G {\mathcal P}_g(S^n)\big)$ and  
\[
\text{Cotor}^{*,*}_{H^*(G)}({\mathbb Z}/p, H^*({\mathcal P}_G(S^n)))= \oplus_{g\in G}
\text{Cotor}^{*,*}_{H^*(G)}({\mathbb Z}/p, H^*({\mathcal P}_g(S^n))). 
\]
We compute the cotorsion functor with the normalized cobar complex 
\[
\big( {\mathbb Z}/p\{\tau^*\}^{\otimes k}\otimes H^*({\mathcal P}_g(S^n)), \partial_k= \nabla_G\otimes 1 + (-1)^{k+1} 1\otimes \nabla_{\tau^*}\big)_{k\geq 0}, 
\] 
where $\tau \in G$ denotes the nontrivial element, $\nabla_{\tau^*} : H^*({\mathcal P}_g(S^n)) \to \widetilde{H}^0(G)\otimes H^*({\mathcal P}_g(S^n))= {\mathbb Z}/p\{\tau^*\}\otimes  H^*({\mathcal P}_g(S^n))$ is the coaction induced by the $G$-action on ${\mathcal P}_g(S^n)$ and 
the projection $G\times {\mathcal P}_g(S^n) \to {\mathcal P}_g(S^n)$ gives rise to the map $\nabla_G$. Observe that the complex is nothing but the $E_1$-term of the spectral sequence $\{E_r^{*,*}, d_r\}$. 

In what follows, we may write $\K$ for the underlying field $ {\mathbb Z}/p$.  
We consider the EMSS associated with the fibre square (\ref{eq:PB_diagram}) for $M =S^n$. 

For the case $n = 2m+1$, the action on $H^*(S^{2m+1}) \cong \wedge (y)$ induced by the nontrivial element $\tau$ is given by 
$\tau^*(y) = y$. Therefore, the computation in \cite[Theorem 2.1]{K-Y} allows us to conclude that the cohomology algebra $H^*({\mathcal P}_g(S^n))$ is isomorphic to $H^*(LS^{2m+1})$ for each $g \in G$. The naturality of the EMSS implies that 
the action by $g^*$ on  $H^*(LS^{2m+1})$ is trivial and then so is the coaction. This yields that $E_2^{0, *}\cong\text{Cotor}^{0,*}_{H^*(G)}({\mathbb Z}/p, H^*({\mathcal P}_G(S^n)))\cong 
H^*(LS^{2m+1})^{\oplus 2}$ and $E_2^{p, *} = 0$ for $p>0$. Then, it follows that $E_2^{*,*} \cong E_\infty^{*,*}$ and there is no extension problem. The explicit form of $H^*(LS^{2m+1})$ follows from \cite[Theorem 2.1]{K-Y}. This enables us to obtain the first assertion. 

We consider the case where $n = 2m$. In order to compute the torsion functor which gives the $E_2$-term of the EMSS, 
we recall a Koszul--Tate resolution of the form 
\[
{\mathcal F} = (\Lambda \otimes \Lambda \otimes \wedge(u)\otimes \Gamma[w], d) \stackrel{\varepsilon}{\to} \Lambda \to 0
\] 
of $\Lambda := H^*(S^{2m})=\K[x]/(x^2)$ as a left $\Lambda \otimes \Lambda$-module, where $\varepsilon$ is the multiplication on $\Lambda$, 
$d(\Lambda \otimes \Lambda) = 0$, $d(u) = x\otimes 1-1\otimes x$, $d(\gamma_r(w))= (x\otimes 1 + 1 \otimes x)u\otimes \gamma_{r-1}(w)$, $\text{bideg}\ u = (-1, \deg x)$ and $\text{bideg}\ \gamma_r(w) = r(-2, 2\deg x)$; see 
\cite[Proposition 3.5]{Smith1} and \cite[Proposition 1.1]{K91}. The complex 
$(\Lambda\otimes _{\Lambda \otimes \Lambda}{\mathcal F}, 1\otimes d)$ computes the Hochschild homology of $\Lambda$. 
Using the resolution, we determine $H^*({\mathcal P}_0(S^n))$ and $H^*({\mathcal P}_\tau(S^n))$ for the nontrivial element $\tau \in G$

\begin{claim}\label{claim:P} \text{\em i)} $H^*({\mathcal P}_0(S^{2m}))\cong 
\{\K[x]/(x^2)\otimes \wedge(u)/ (xu)_A\}\oplus\{(x, u)_A/(xu)_A \}\otimes \Gamma^+[w]$ as an algebra, where $( S )_A$ denotes the ideal of 
$A:= \K[x]/(x^2)\otimes \wedge(u)$ generated by a set $S$. Moreover, $\tau^*(z) = -z$ for $z \in \widetilde{H}^*({\mathcal P}_0(S^n))$. \\
\text{\em ii)} $H^*({\mathcal P}_\tau(S^{2m}))\cong \wedge(x\otimes u)\otimes \Gamma[w]$ as an algebra  
and $\tau^*(z) = z$ for each element 
$z \in  \wedge(x\otimes u)\otimes \Gamma[w]$. 
\end{claim}
It follows from Claim \ref{claim:P} i) that $\text{Cotor}^{*,*}_{H^*(G)}({\mathbb Z}/p, H^*({\mathcal P}_0(S^{2m})))= {\mathbb Z}/p$. Moreover,  
Claim \ref{claim:P}  ii) enables us to conclude that 
\begin{eqnarray*}
\text{Cotor}^{i,*}_{H^*(G)}({\mathbb Z}/p, H^*({\mathcal P}_\tau(S^{2m})))\cong \left\{
\begin{array}{ll}
\wedge(x\otimes u)\otimes \Gamma[w] & \text{for $i = 0$} \\ 
0 & \text{for $i \neq 0$}
\end{array}
\right.
\end{eqnarray*}
We have the result. 
\end{proof}

\begin{proof}[Proof of Claim \ref{claim:P}] Let $n$ be an even integer $2m$.
i)  Since ${\mathcal P}_0(S^n)$ is nothing but the free loop space $LS^n$, the result on the algebra structure 
follows from the results in \cite[4.1]{N-T} and \cite[Theorem 2.2]{K-Y}. In order to prove the latter assertion on the action, we represent elements in the Koszul--Tate resolution with the bar resolution ${\mathbb B}^{*,*}$ 
of $\Lambda$ as a left $\Lambda\otimes \Lambda$-module. 
Let $\psi : {\mathbb B}^{*,*} \to {\mathcal F}$ be the chain map constructed in the proof of \cite[Lemma 1.5]{K91}, which induces an isomorphism between the torsion functors. We see that $\psi(x) = x$, 
$\psi(u) = [x\otimes 1-1\otimes x]$ and  
\[
\Psi (1_{\Lambda\otimes\Lambda}[x\otimes 1+1 \otimes x \mid x\otimes 1-1 \otimes x \mid \cdots  \mid x\otimes 1+1 \otimes x \mid x\otimes 1-1 \otimes x]1_\Lambda) =
\gamma_r(w). 
\]
Since $\tau^*(x) = -x$ for $x \in H^*(S^n)$, 
it follows that $\tau^*(u)= -u$ and $\tau^*(\gamma_r(w)) = \gamma_r(w)$. Consider the morphism of spectral sequences induced by the diagram (\ref{eq:actions}). Then, the naturality of the EMSS and the forms of algebra generators of $H^*({\mathcal P}_0(S^n))$ enable us to obtain the latter half of i). 

ii) To compute the cohomology algebra $H^*({\mathcal P}_\tau(S^n))$ for the nontrivial element $\tau\in G$, we consider 
the EMSS $\{\widetilde{E}_r^{*,*}, \widetilde{d}_r\}$ associated 
with the pullback diagram (\ref{eq:PB_diagram}) 
converging to the cohomology $H^*({\mathcal P}_\tau(S^n))$. 
As mentioned in the proof of i),  the $\tau$-action on $H^{n}(S^{n})$ is nothing but the multiplication by $-1$. 
Thus, 
the right $\Lambda \otimes \Lambda$-module structure on $\Lambda_\tau := H^*(S^{n})$ is given by 
$a\cdot (\lambda\otimes \lambda') = - a\lambda \lambda'$ for $a \in \Lambda_\tau$ and $\lambda\otimes \lambda \in \Lambda \otimes \Lambda$. This yields that 
\[
\widetilde{E}_2^{*,*} \cong \text{Tor}_{\Lambda \otimes \Lambda}(\Lambda_\tau, \Lambda) \cong 
(\Lambda_\tau \otimes_{\Lambda\otimes \Lambda} {\mathcal F}, d'(u)= 2x, d'(\gamma_r(w))= 0)\cong \wedge(x\otimes u)\otimes \Gamma[w]
\]
as bigraded algebras, where $\text{bideg} (x\otimes u) = (-1, 2n)$. For degree reasons, we see that the EMSS collapses at the $E_2$-term. In fact, the possibility of 
a nontrivial differential appears as $d_s(\gamma_{p^f}(w)) = \alpha x\otimes u \cdot \gamma_l(w)$ for some positive integers $s$, $f$, $l$ and some $\alpha \in \K$. We observe that $s\geq 2$. Comparing the total degrees of both elements in the equality, we have $(2n-1)p^f +1 = (2n-1) + (2n-2)l$ and then $p^f = l+1$. By comparing the filtration degree, we see that $-2p^f +s  = -1 -2(p^f-1)$ and $s=1$, which is a contradiction. 

We have to show that $\gamma_p^f(w)^p =0$ in $H^*({\mathcal P}_\tau(S^n))$ for $f\geq 0$. 
The extension problems are solved by degree reasons. It turns out that  $H^*({\mathcal P}_\tau(S^n))\cong \text{Tot}\widetilde{E}_\infty^{*,*} \cong \text{Tot}\widetilde{E}_\infty^{*,*}$ as algebras. 

By using the chain map $\Psi : {\mathbb B}^{*,*} \to {\mathcal F}$ as mentioned above, we see that 
$\tau^*(x\otimes u) = (-1)(-1)x\otimes u$ and $\tau^*(\gamma_r(w)) = \gamma_r(w)$. Thus,  by considering again 
the morphism of spectral sequences induced by the diagram (\ref{eq:actions}), we have the result on the action.  
\end{proof}

\begin{rem} The path space ${\mathcal P}_0(S^{n})$ is nothing but the free loop space $LS^n$. Then, the inclusion 
$LS^n= {\mathcal P}_0(S^{n}) \to EG \times {\mathcal P}_0(S^{n})$ defines a natural map 
$\alpha : LS^n \to EG\times_G {\mathcal P}_0(S^{n})$. Moreover, 
we have a commutative diagram 
\[
\xymatrix@C20pt@R10pt{
LS^n \ar[r]_-{Lp} \ar@/^1.5pc/[rrr]^{\alpha} & L\R P^n & \coprod_{g \in G}(EG \times_G {\mathcal P}_g(S^{n}))  \ar[l]_-{\overline{p}}^-{\simeq_w} & EG\times_G {\mathcal P}_0(S^{n}) \ar[l]\\
S^n \ar[u]^-s  \ar[r]^-{p} & \R P^n \ar[u]_-s, 
}
\]
where $p$ is the universal cover and $s$ denotes the map that assigns the constant loop at $r$ to each element $r$ in $S^n$ and 
$\R P^n$, respectively.  Thus, it follows that $LS^n$ is weak homotopy equivalent to the path component of $L\R P^n$ consisting of constant loops under the equivalence (\ref{eq:w}). In particular, the proof of Theorem \ref{thm:computation_I} allows us to conclude that the cohomology of the component 
with coefficients in ${\mathbb Z}/p$ 
is isomorphic to ${\mathbb Z}/p$ if $n$ is even. 
\end{rem}

We here attempt to compute the cohomology of the free loop space of a Borel construction with coefficients in $\mathbb{Z}/2$ by using the same method as above. We indeed use the multiplicative structure of the spectral sequence that we apply.  

Let $G$ be the cyclic group $\mathbb{Z}/2$ and $M$ a simply-connected, mod $2$ homology $n$-sphere that admits a $G$-action, where $n \geq 2$.  
Further assume that the Borel cohomology 
$H^*(EG\times_G M; {\mathbb Z}/2)$ is of finite dimension. Let $L_0(EG\times_G M)$ be the connected component of the free loop space 
$L(EG\times_G M)$ containing the constant loops.  
Then, we have the following result. 
\begin{prop}\label{prop:L_0} 
As an algebra,
$
H^*(L_0(EG\times_G M); {\mathbb Z}/2) \cong \Gamma[y]\otimes \big({\mathbb Z}/2[t]/(t^{n+1})\big)
$
for $* \leq 2n -3$, where $\deg y = n-1$ and $\deg  t = 1$. 
\end{prop}

\begin{proof}
To obtain the isomorphism, 
we compare spectral sequences given by the groupoid 
$
\G_1:=\xymatrix@C15pt@R5pt{[G\times {\mathcal P}_0(EG \times M)  \ar@<-0.5ex>[r] \ar@<0.5ex>[r]& {\mathcal P}_0(EG\times M) ]}
$
and the translation groupoid $\G_2:=\xymatrix@C15pt@R5pt{[G\times (EG\times M)  \ar@<-0.5ex>[r] \ar@<0.5ex>[r]& EG\times M]}$\!, respectively. 
Let $\widetilde{ev_0} : \G_1 \to \G_2$ be the morphism of groupoids induced by the evaluation map 
$ev_0 : {\mathcal P}_0(EG\times M) \to EG \times M$ at $0\in I$, where $ev_0$ is the evaluation map at $0$.  
 Observe that $ev_0$ is a $G$-equivariant map.  
Thus, we have a commutative diagram 
\[
\xymatrix@C25pt@R18pt{
EG\times_G {\mathcal P}_0(EG \times M) \ar[d]_{1\times ev_0} \ar[r]^-{\overline{p}}_-{\simeq_w}& L_0(EG\times_G M)\ar[d]^{ev_0} \\
EG\times_G (EG \times M)  \ar[r]_-{\simeq_w}^-{q} & EG\times_G M, 
}
\]
where $\overline{p}$ is the weak equivalence described in Proposition \ref{prop:An_orbifold_stack} and 
$q$ is induced by the natural projection.  
Let $\{{}_1E_r^{*,*}, {}_1d_r\}$ and  $\{{}_2E_r^{*,*}, {}_2d_r\}$ be the spectral sequences in Theorem \ref{thm:App} ii) with coefficients in ${\mathbb Z}/2$ constructed by groupoids $\G_1$ and $\G_2$, respectively. 
We observe that ${\mathcal P}_0(EG \times M) = L(EG\times M) \simeq LM$. 

By the assumption on $M$, 
we see that $H^*(LM)\cong {\mathbb Z}/2[x_n]/(x_n^2)\otimes \Gamma[y]$ as an algebra, where $\deg x_n = n$ and $\deg y = n-1$. In fact, the isomorphism for $n$ odd follows from \cite[Theorem 2.1]{K-Y}. Moreover, the computation in \cite[4.1]{N-T} gives the result for $n$ even. 
Indeed, we need an explicit form of the mod $2$ cohomology group of $\Omega M$ as an input datum; that is, $H^*(\Omega M; {\mathbb Z}/2)\cong \Gamma[y]$, where $\deg y = n-1$. This fact follows from the computation of the EMSS converging to $H^*(\Omega M; {\mathbb Z}/2)$ with $E_2^{,*}\cong \text{Tor}_{H^*(M)}^{*,*}({\mathbb Z}/2, {\mathbb Z}/2)$. 

Since each dimension of $A^i:=H^i({\mathcal P}_0(EG \times M))$ is one or zero. Thus, the $G$-action on $A^i$ is trivial. This fact and the formula (\ref{eq:cupPD}) allow us to deduce that  
\[
{}_1E_2^{*,*} \cong \text{Cotor}_{H^*(G)}^{*,*}({\mathbb Z}/2, A^*)\cong A^*\otimes {\mathbb Z}/2[t] \cong {\mathbb Z}/2[x_n]/(x_n^2)\otimes \Gamma[y] \otimes {\mathbb Z}/2[t]
\]
as algebras, where $\text{bideg}\ t = (1, 0)$. The same argument as above is applicable to the spectral sequence $\{{}_2E_r^{*,*}, {}_2d_r\}$. Then, it follows  that 
\[
{}_2E_2^{*,*} \cong \text{Cotor}_{H^*(G)}^{*,*}({\mathbb Z}/2, H^*(M))\cong  {\mathbb Z}/2[x_n]/(x_n^2)\otimes {\mathbb Z}/2[t]
\] 
as algebras. 

We consider the morphism $\{f_r\} : \{{}_2E_r^{*,*}, {}_2d_r\} \to \{{}_1E_r^{*,*}, {}_2d_r\}$ of spectral sequences induced by 
$\widetilde{ev_0}$. The construction of the spectral sequence yields that  $f_2(x_n) = x_n$ and $f_2(t)=t$. By assumption, the vector space $H^*(EG\times_G M)$ is of finite dimensional. Therefore, we see that  ${}_2d_{n+1}(x_n) = t^{n+1}$ and hence ${}_1d_{n+1}(x_n) = t^{n+1}$. To complete the proof, it remains to show that ${}_1d_{n-1}(y) = 0$ but not ${}_1d_n(y) = t^n$. 
Since $ev_0 :  L_0(EG\times_G M) \to EG\times_G M$ has a section, it follows that the map $ev_0^*$ induced on the cohomology is a monomorphism. 
This implies that ${}_1d_{n-1}(y) = 0$. 
\end{proof}

\begin{rem} In the spectral sequence $\{{}_1E_r^{*,*}, {}_1d_r\}$ above, 
there is a possibility that ${}_1d_n(\gamma_2(\overline{x_n}))= \overline{x_n}\otimes t^n$. 
In the computation above, we can use the Leray--Serre spectral sequence for a Borel fibration of the form 
$X \to EG\times_G X \to BG$ instead of the spectral sequence in Theorem \ref{thm:App} ii). 
However, the same possibility of the non-trivial differential as above remains. 
\end{rem}

\begin{rem} We observe that Proposition \ref{prop:L_0} is applicable to the real projective space. Let $M$ be the $3$-dimensional real projective space ${\mathbb R}P^3$ which is regarded as the homogeneous space of $S^3=SU(2)$ by the center ${\mathbb Z}/2$. Since $M$ is an H-space, 
it follows that $L(M) \simeq \Omega_{[e]} M \times M \simeq (\Omega_e S^3 \coprod \Omega_e S^3) \times M$, where $e\in S^3$ denotes the unit; see  the proof of Proposition \ref{prop:An_orbifold_stack} for the second weak homotopy equivalence.  Then it turns out that 
$H^*(LM; {\mathbb Z}/2)\cong (\Gamma[s^{-1}x])^{\oplus 2}\otimes \big({\mathbb Z}/2[t]/(t^{4})\big)$ as an algebra. 
\end{rem}

\section{A spectral sequence for a Borel construction}\label{GlobalQuotients}
In the previous section, we apply the spectral sequence in Theorem \ref{thm:App} ii) to a computation of the cohomology of a transformation groupoid. 
In this section, the result is generalized with a description of the cotorsion functor by the cohomology of a finite group. 
Moreover, we give a spectral sequence for computing the cohomology of the classifying space of {\it inertia} groupoids, while there is no computational example. 

We begin with a lemma relating the cotorsion functor to the cohomology of a group.   
Let $G$ be a finite group and $N$ a finite-dimensional left $\K[G]^\vee$-comodule. Then 
the module $N^\vee$ is regarded as a left $\K[G]$-module via the natural map
$\nabla^\vee  :  \K[G]\otimes N^\vee \to N^\vee$ induced by the left comodule structure $\nabla$ on $N$. 
Thus, by using the isomorphism $N \cong (N^\vee)^\vee$, we consider $N$ a right $\K[G]$-module. 

For a Hopf algebra $A$ with antipode $S$, we have an injective algebra homomorphism $\delta : A \to A^{e}= A\otimes A^{\text{op}} $ defined by 
$\delta(a) = \sum a_1\otimes S(a_2)$; see \cite[Lemma 9.4.1]{Wh}. 

\begin{lem}\label{lem:Cotor_G} 
Under the setup above, 
there are isomorphisms of vector spaces 
\[
\text{\em Cotor}^*_{\K[G]^\vee}(\K, N) \cong \text{\em HH}^*(\K[G], \K\otimes N) \cong \text{\em Ext}^*_{\K[G]}(\K, \K\otimes N)= H^*(G, N). 
\] 
Here the module $N$ in the Hochschild cohomology is regarded as a right $\K[G]$-module mentioned above. 
Moreover, the module $\K\otimes N$ in the group cohomology is considered a left $\K[G]$-module via the monomorphism $\delta$. 
\end{lem}

\begin{proof}
The first isomorphism follows from \cite[Theorem 3.4]{A-W}. As for the second isomorphism, we repeat the proof of 
\cite[Theorem 9.4.5]{Wh} verbatim. For a general Hopf algebra $A$, the result \cite[Lemma 9.4.2]{Wh} enables us to conclude that 
the map $f : A \to A^e\otimes_A \K$ defined by $f(a) = a\otimes 1\otimes 1$ is an isomorphism of $A^e$-modules, where $A^e$ denotes the enveloping algebra of $A$. Then, it follows  
 that $\text{Ext}^n_{A^e}(A, T)\cong \text{Ext}^n_{A^e}(A^e\otimes_A\K, T)\cong 
\text{Ext}_A^n(\K, T)$ for a left $A^e$-module $T$. We observe that the latter isomorphism is given by Eckmann-Shapiro Lemma; see, for example, 
\cite[Lemma A.6.2]{Wh}. This completes the proof. 
\end{proof}

We can generalize the computation in Section \ref{sect:example(s)}.  

\begin{thm}\label{thm:L_Cotensor} \text{\em (cf. \cite[Chapter II, Theorem 19.2]{Bredon})} 
Let $G$ be a finite group acting on a space $X$ whose cohomology with coefficients in $\K$ is locally finite. 
Suppose that $\text{\em ch}(\K)$ the characteristic of $\K$ is coprime with 
$|G|$. Then as an algebra 
\[
H^*(EG\times_G X; \K)\cong \K\Box_{\K[G]^\vee}H^*(X; \K), 
\]
where the right-hand side algebra denotes the cotensor product of the trivial right $\K[G]^\vee$-comodule $\K$ and the left 
$\K[G]^\vee$-comodule $H^*(X; \K)$. 
\end{thm}

\begin{proof} 
In view of Theorem \ref{thm:App} ii), we have a spectral sequence $\{E_r^{*,*}, d_r\}$ converging to 
$H^*(EG\times_G X)$ with $E_2^{*,*} \cong\text{Cotor}_{H^*(G)}^{*,*}(\K, H^*(X))$ as an algebra. Since $H^*(G) =H^0(G) =\K[G]^\vee$ 
and the characteristic of $\K$ is coprime with 
$|G|$, it follows from Lemma \ref{lem:Cotor_G} that 
\begin{eqnarray*}
           \text{Cotor}^{p,q}_{H^*(G)}(\K, H^*(X)) \cong 
\text{Cotor}_{\K[G]^\vee}^p(\K, H^q(X)) =
\left\{
\begin{array}{ll}
\K\Box_{\K[G]^\vee}H^q(X) & \text{for $p = 0$} \\ 
0 & \text{for $p \neq 0$}.
\end{array}
\right.
\end{eqnarray*}
We see that  the coalgebra structure $\nabla : H^*(X) \to H^*(G)\otimes H^*(X)$ is a morphism of algebras. Then, the cotensor product $\K\Box_{\K[G]^\vee}H^*(X)$ is a subalgebra of $H^*(X)$. It turns out that the spectral sequence of algebras collapses at the $E_2$-term and there is no extension problem. We have the result. \end{proof}

\begin{rem}\label{rem:naturality}
Let  $G$ be a finite group acting on a space $X$. Then, the inclusion from the unit to $G$ gives rise to a morphism $\xymatrix@C12pt@R5pt{\iota : [\{e\}\times X \ar@<-0.5ex>[r] \ar@<0.5ex>[r]& X] \ar[r] & [G\times X \ar@<-0.5ex>[r] \ar@<0.5ex>[r]& X]}$ of the translation groupoids, whose domain is the trivial one. The naturality of the spectral sequence in Theorem \ref{thm:main}  and  the proof of Theorem \ref{thm:App} ii) enable us to obtain a morphism 
$\{\iota^*_r\} : \{E_r^{*,*}, d_r\} \to  \{{}^\backprime E_r^{*,*}, {}^\backprime d_r\}$  
of spectral sequences and a commutative diagram 
 \[
\xymatrix@C25pt@R18pt{
 H^*(EG\times_G X;K) \ar[d]_{j^*}  \ar@{->>}[r]& E_\infty^{0, *} \ar[d]_{\iota^*_\infty} \  \ar@{>->}[r]& E_2^{0, *}\ar[d]_{\iota^*_2} \ar[r]^-{\cong} & \K\Box_{\K[G]^\vee}H^*(X; \K)  \ar[d]^{inclusion} \\
 H^*(X; \K) \ar[r]_{=}&  {}^\backprime E_\infty^{0, *}  \ar[r]_{=} &  {}^\backprime E_2^{0, *} \ar[r]_{=}&  H^*(X; \K),
}
\]
where $j$ is the composite $X \to E\{e\}\times X \to EG\times X \to EG\times_G X$. 
The proof of Theorem \ref{thm:L_Cotensor} shows that the upper arrows are isomorphisms if the characteristic of $\K$ is coprime with $|G|$. Then, the isomorphism of algebras in the theorem is regarded as the map induced by $j : X\to EG\times_G X$ mentioned above. 
\end{rem}

Theorem \ref{thm:L_Cotensor} allows us to compute the cohomology algebra of the free loop space of a Borel construction.  

\begin{cor}\label{cor:L_Cotensor}\text{\em (cf. \cite[Corollary 6.5]{L-U-X})}
Let $G$ be a finite group acting on a space $M$ and ${\mathcal P}_G(M)$ the space defined in (\ref{eq:P_G}). 
Suppose that $H^*({\mathcal P}_G(M); \K)$ is locally finite and $(\text{\em ch}(\K), |G|) =1$. Then as an algebra 
\[
H^*(L(EG\times_G M); \K)\cong \K\Box_{\K[G]^\vee}H^*({\mathcal P}_G(M); \K).
\]
\end{cor}

\begin{proof}
Theorem \ref{thm:L_Cotensor} and Corollary \ref{cor:L_Borel} imply the assertion. 
\end{proof}

\begin{rem}\label{rem:locally_finiteM}
Let $M$ be a simply-connected $G$-space whose cohomology with coefficients in a field $\K$ is locally finite. 
Then, the total complex of the bar complex, which computes the $E_2$-term of the EMSS for the fibre square (\ref{eq:PB_diagram}), 
is locally finite and hence   
so is  $H^*({\mathcal P}_G(M); \K)$ if $G$ is finite. 
\end{rem}

\begin{rem}\label{rem:L(M/G)}
Let $G \to M \stackrel{p}{\to} M/G$ be a principal $G$-bundle with finite fibre $G$. The projection $p$ induces natural 
maps $\widetilde{p} :  {\mathcal P}_G(M) \to  L(M/G)$ and $\overline{p} : EG\times_G  {\mathcal P}_G(M) \to L(M/G)$ 
with $\overline{p} \circ j = \widetilde{p}$, where 
$j$ is the map described in Remark \ref{rem:naturality}. 
By virtue of Proposition \ref{prop:An_orbifold_stack} and Theorem \ref{thm:L_Cotensor}, we see that 
$\widetilde{p}$ gives rise to an isomorphism 
$\widetilde{p}^* : H^*(L(M/G); \K)\stackrel{\cong}{\longrightarrow}  K\Box_{\K[G]^\vee}H^*({\mathcal P}_G(M); \K)$ of algebras if 
$(\text{ch}(\K), {G}) =1$.
\end{rem}

\begin{cor}\label{cor:L_Cotensor_HH}
Let $G$ be a finite group and $M$ a simply-connected $G$-space whose rational cohomology is locally finite. 
Assume further that $M$ is formal and for each $g\in G$, the map $g : M \to M$ induced by $g$ via the $G$-action is formal in the sense of \cite[D\'efinition 2.3.3]{V-P}; 
see also \cite[Definition 2.3]{A}. 
Then as an algebra 
\[
H^*(L(EG\times_G M); \Q)\cong \Q\Box_{\Q[G]^\vee}\big(\oplus_{g \in G} HH_*(H^*(M; \Q), H^*(M; \Q)_g)\big ), 
\]
where $HH_*(A, H)$ denotes the Hochschild homology of a graded commutative algebra $A$ with coefficients in an $A^e$-algebra $H$.
\end{cor}

\begin{proof} 
The assumption on the map $g$ implies that the map $1\times g : M \to M\times M$ is formal. In fact, 
let ${\mathcal M}_M \stackrel{\simeq}{\to} A_{PL}^*(M)$ be a minimal model for $M$, where $A_{PL}(M)$ denotes the complex of polynomial forms on $M$; see Appendix \ref{sect:S-deRham}. 
Suppose that $g$ is $(\varphi_M)$-$(\varphi_M')$-formal with 
quasi-isomorphisms $\varphi_M : {\mathcal M}_M \stackrel{\simeq}{\to} H^*({\mathcal M}_M)$ and 
$\varphi_M' : {\mathcal M}_M \stackrel{\simeq}{\to} H^*({\mathcal M}_M)$. Then the map $1\times g$ is $(\varphi_M\otimes \varphi_M')$-$(\varphi_M')$-formal. This fact  
enables us to deduce that 
\[H^*({\mathcal P}_g(M); \Q)\cong HH_*(A_{PL}^*(M), A_{PL}^*(M)_g\big )\cong HH_*(H^*(M; \Q), H^*(M; \Q)_g\big )
\] 
as algebras. Corollary \ref{cor:L_Cotensor} yields the result. 
\end{proof}

We consider a spectral sequence for the inertia groupoid associated with a topological groupoid 
$\G$; see \cite[Chapters 11 and 16]{BGNX} for such a groupoid and its importance in string topology for stacks. In particular, we have a canonical morphism from the inertia groupoid of $\G$ to the loop groupoid; see \cite[12.4]{BGNX}. 

Let $\G=\xymatrix@C15pt@R5pt{[G_1 \ar@<-0.5ex>[r] \ar@<0.5ex>[r]& G_0]}$ be a topological groupoid. By definition, a space $X$ with a map $p : X \to G_0$ is a 
{\it $\G$-space} if $X$ is endowed with an action $G_1{}^t\times^p_{G_0}X \to X$, where $G_1{}^t\!\times^p_{G_0}X$ denotes the pullback of $p$ along the target map $t : G_1\to G_0$. Then, the {\it translation groupoid} associated with the $\G$-space $X$ is defined by 
$\G \ltimes X := \xymatrix@C15pt@R5pt{[G_1{}^t\!\times^p_{G_0}X \ar@<-0.5ex>[r] \ar@<0.5ex>[r]& X]}$ whose source and target maps are defined by the projection in the second factor and the action, respectively.  

For a topological groupoid $\G$, we define a subspace 
$S_{\G}$ of $G_1$  and map $p : S_{\G} \to G_0$ by $S_{\G}:=\{ g \in G_1 \mid s(g) = t(g) \}$ 
and $p(g) = s(g)$, respectively. Then, the conjugation on $G_1$ gives rise to an action 
$G_1 {}^t\!\times_{G_0}^p S_{\G} \to S_{\G}$. The translation groupoid $\Lambda \G := \G \ltimes S_{\G}$ is 
called the {\it inertia groupoid} associated with $\G$. We also define the {evaluation map} 
\[
ev_0 :  \Lambda \G =\xymatrix@C15pt@R5pt{[G_1 {}^t\!\times_{G_0}^p S_{\G} \ar@<-0.5ex>[r] \ar@<0.5ex>[r]& S_\G]} 
\to \G=\xymatrix@C15pt@R5pt{[G_1 \ar@<-0.5ex>[r] \ar@<0.5ex>[r]& G_0]}
\] 
by the projection in the first factor in morphisms and by the source map of $\G$ in the objects. 

Let $G$ be a finite group and $G \ltimes M := \xymatrix@C15pt@R5pt{[G\times M \ar@<-0.5ex>[r] \ar@<0.5ex>[r]& M]}$ a translation groupoid. Applying the construction above, we have the inertia groupoid $\Lambda (G \ltimes M)$. Moreover, we see that the groupid has the form 
\[
\Lambda(G \ltimes M) = \xymatrix@C15pt@R5pt{[G\times \coprod_{g\in G}(M^g \times \{g\}) \ar@<-0.5ex>[r] \ar@<0.5ex>[r]& \coprod_{g\in G}(M^g\times \{g\})]}, 
\]
where $M^g$ denotes the space of fixed points of $g$. 
Observe that the source map is the projection of the second factor and the conjugation gives rise to the target map $t$, more precisely,  
$t(h, ( m, g))= (h m, hgh^{-1})$; see \cite[(17.2.1)]{BGNX}.  By virtue of Theorem  \ref{thm:App} ii), we have the following result.  

\begin{prop}\label{prop:ss} 
Let $G \ltimes M$ be a transformation groupoid for which $G$ is a finite group and $H^*(M^g)$ is locally finite for each $g \in G$. Then 
there exists a spectral sequence $\{E_r, d_r\}$ converging to 
the cohomology of the classifying space $\text{\em B}\Lambda (G \ltimes M)$ of the inertia groupoid $\Lambda (G \ltimes M)$, 
as an algebra with 
\[
E_2^{p,q}\cong \text{\em Cotor}^{p, q}_{H^0(G)}(\K ,\oplus_{g\in G}H^*(M^g; \K)).
\]
Moreover, the evaluation map gives rise to a morphism from the spectral sequence in Theorem \ref{thm:App} ii) to $\{E_r, d_r\}$. Assume further that $(\text{\em ch}(\K), |G|) =1$, then one has 
$H^*(B\Lambda (G \ltimes M); \K)\cong \K\Box_{\K[G]^\vee}(\oplus_{g\in G}H^*(M^g; \K))$ as an algebra. 
\end{prop}


Let $G$ be a discrete group. We consider the free loop space $LBG$ of the classifying space $BG = EG/G= EG\times_G\ast$. 
\begin{prop}\label{prop:LBG} \text{\em (cf. \cite[Example 6.2]{L-U-X} and \cite[Proposition 2.12.2]{Benson})} Let $G$ be a finite group. Then as an algebra, 
\[
H^*(LBG; \K) \cong \text{\em Cotor}_{\K[G]^\vee}^*(\K, (\K[G]_{\text{\em ad}})^\vee), 
\]
where $\K[G]_{\text{\em ad}}$ denotes the adjoint representation of $G$. Especially, if $G$ is abelian, then 
$H^*(LBG; \K) \cong H^*(G; \K)^{\oplus |G|}$ as an algebra. 
\end{prop}

\begin{proof} 
It follows that $H_*({\mathcal P}_g(\ast); \K) =\K$ for each $g \in G$. By the definition of the $G$-action on 
${\mathcal P}_G(\ast)$, we see that  $H_*({\mathcal P}_G(\ast); \K)\cong \K[G]_{\text{ad}}$. Then, the weak equivalence in $(\ref{eq:w})$ and the spectral sequence $\{E_r^{*,*}, d_r\}$ in Theorem \ref{thm:App} ii) enable us to deduce the first assertion. Observe that $E_2^{*,q} = 0$ for $q > 0$. 

The latter half follows from the fact that the $G$-action on $\K[G]_{\text{ad}}$ is trivial. In particular, 
the explicit formula (\ref{eq:cupPD}) of the multiplication on the cotorsion product 
yields the result on the algebra structure. 
\end{proof}

\begin{rem} 
If $G$ is abelian, then the classifying space $BG$ is an H-space. This allows us to obtain a homotopy equivalence $BG\times \Omega BG \stackrel{\simeq}{\to} LBG$ induced by the product on $BG$. The latter half of Proposition \ref{prop:LBG} also follows from the fact. 
\end{rem}

We conclude this section with comments on a spectral sequence converging to the free loop space of a Borel construction, which is obtained by Theorem \ref{thm:App} ii). 

\begin{rem}\label{rem:SS_L}
Let $G$ be a finite group with an action on a space $M$ and $\{E_r^{*,*}, d_r\}$ the spectral sequence in Theorem \ref{thm:App} ii) for 
the Borel construction $EG \times_G {\mathcal P}_G(M)$. By Corollary \ref{cor:L_Borel}, we see that 
$L(EG\times_G M)$ is connected to $EG\times_G{\mathcal P}_G(M)$ with weak homotopy equivalences. Thus, 
under the same condition as in Corollary \ref{cor:L_Cotensor} but $(\text{ch}(\K), |G|) =1$, the spectral sequence converges to $H^*(L(EG\times_G M); \K)$ as an algebra. 

The form of the $E_2$-term enables us to deduce that the vertical edge ($q$-axis) $E_2^{0,*}$ is isomorphic to the cotensor product $\K\Box_{\K[G]^\vee}H^*({\mathcal P}_G(M); \K)$, which is a subspace of $H^*({\mathcal P}_G(M); \K)$ detected by the transfer homomorphism if $(\text{ch}(\K), |G|) =1$; see \cite[Chapter II, Theorem 19.2]{Bredon} and Lemma \ref{lem:Cotor_G}. 

We consider the horizontal edge ($p$-axis). Since the trivial map $v : {\mathcal P}_g(M) \to \ast$ gives a $G$-equivariant map $H_0({\mathcal P}_G(M); \K) \to \K[G]_{\text{ad}}$, 
it follows from Proposition \ref{prop:LBG} that the edge $E_2^{*,0}$  is a module over the cotorsion product 
$\text{Cotor}_{\K[G]^\vee}^*(\K, (\K[G]_{\text{ad}})^\vee)$ and hence $H^*(LBG; \K)$ via the morphism 
$$v^* : H^*(LBG; \K) \to \text{Cotor}_{\K[G]^\vee}^{*, 0}(\K, H^*({\mathcal P}_G(M); \K))$$ of algebras induced by $v$. 
Observe that the $v^*$ is an isomorphism of algebras if $M$ is simply-connected. In fact, we see that each ${\mathcal P}_g(M)$ is connected in that case.  
We remark that the underlying algebra $H^*(LBG; \K)$ is indeed the Hochschild cohomology of the group ring $\K[G]$; see \cite[Proposition 2.12.2]{Benson}. 

Each differential on $E_r^{*,0}$ for $r\geq 2$ is trivial. Therefore, the multiplicative structure on the spectral sequence gives rise to an $H^*(LBG; \K)$-module structure on the spectral sequence; that is, each term $E_r^{*,*}$ admits an $H^*(LBG; \K)$-module structure defined by 
\[
\xymatrix@C20pt@R10pt{
\bullet : H^p(LBG; \K)\otimes E_r^{*, *} \ar[r]^-{v^*\otimes 1} & E_2^{p, 0}\otimes E_r^{*,*} \ar@{->>}[r]^-{p_r\otimes 1} &  E_r^{p, 0}\otimes E_r^{*, *} \ar[r]^-{m} &  E_r^{*+p, *},
}
\]
where $p_r$ is the canonical projection and $m$ is the product structure on the $E_r$-term. Thus we see that $d_r(a\bullet x) = (-1)^p a\bullet d_r(x)$ for $a \in H^p(LBG; \K)$ and 
$x \in E_r^{*, *}$. 
\end{rem}

\section{The de Rham cohomology of the diffeological free loop space of a non-simply connected manifold}\label{sect:deRham}

Let $G$ be a finite group acting freely and smoothly on a manifold $M$. Then, a principal $G$-bundle of the form $G \to M \stackrel{p}{\to} M/G$ in the category of manifolds is obtained. 
Let ${\mathcal P}_G^\infty (M)$ be the diffeological space obtained by applying the construction (\ref{eq:P_G}) in $\mathsf{Diff}$ the category of diffeological spaces; see Appendix \ref{sect:App}. We consider a smooth map $\widetilde{p} : {\mathcal P}_G^\infty (M)\to L^\infty (M/G)$ defined by $\widetilde{p}((\gamma, g))= p\circ \gamma$, where 
$L^\infty (M/G)$ denotes the diffeological free loop space of $M/G$; see the pullback diagram (\ref{eq:PB_diagram}).  

For a diffeological space $X$, we may write $H^*_{DR}(X)$ for the singular de Rham cohomology $H^*(A_{DR}(S^D_\bullet(X)))$; see Appendix \ref{sect:App}.  
The purpose of this section is to prove the following theorem. 

\begin{thm}\label{thm:DR} Under the same setting as above, 
suppose further that $M$ is simply connected. Then, the smooth map $\widetilde{p}$ 
gives rise to a well-defined isomorphism 
$$\widetilde{p}^* : H_{DR}^*(L^\infty (M/G)) \stackrel{\cong}{\to} \R\Box_{\R[G]^\vee}H^*_{DR}({\mathcal P}_G^\infty(M))$$ of algebras.  
\end{thm}

Each component ${\mathcal P}_g^\infty(M)$ of ${\mathcal P}_G^\infty(M)$ is constructed with a pullback of the form (\ref{eq:PB_diagram}). Then, by utilizing Theorem \ref{thm:general_main}, 
we may represent algebra generators in the de Rham cohomology $H^*_{DR}(L^\infty(M/G))$ with differential forms on $M$ via Chen's iterated integral map and the {\it factor} map, which connects the Souriau-de Rham complex and the singular de Rham complex; see Appendix \ref{sect:S-deRham}. 
Indeed, the idea is realized in Theorem \ref{thm:deRhamP}.  

Before proving Theorem \ref{thm:DR}, 
we recall the smoothing theorem due to Kihara for a particular case. 
Let $M$ and $N$ be diffeological spaces and $C^\infty(M, N)$ the space of smooth maps from $M$ to $N$ with the functional diffeology. 
We recall the functor $D : \mathsf{Diff} \to \mathsf{Top}$ from Appendix \ref{sect:App}.
Then,  we see that $D(\Delta^n_{st})$ is homeomorphic to $\Delta^n$ the standard $n$ simplex which is a subspace of $\R^{n+1}$; see \cite{Kihara}. Moreover, 
the inclusion  
$i : D(C^\infty(M, N)) \to C^0(DM, DN)$ is continuous; see \cite[Proposition 4.2]{C-S-W}. Thus, it follows that the functor $D$ induces a morphism $\xi :  S^D_\bullet(C^\infty(M, N)) \to \text{Sing}_\bullet(C^0(DM, DN))$ of simplicial sets. 

\begin{thm}\label{thm:smoothing} \text{\em (Smoothing theorem \cite[Theorems 1.1 and 1.7]{Kihara})} Let $M$ and $N$ be 
finite-dimensional manifolds. 
Then, the well-defined map 
$$\xi :  S^D_\bullet(C^\infty(M, N)) \to \text{\em Sing}_\bullet(C^0(DM, DN))$$ is a weak homotopy equivalence. 
\end{thm}

For a manifold $M$, we consider the composite 
\[
\xymatrix@C20pt@R15pt{
\lambda:=i\circ (Dj) : D({\mathcal P}_g^\infty(M)) \ar[r] &  D(C^\infty([0, 1], M)) \ar[r] &C^0(D[0, 1], DM), 
}
\]
where $j$ denotes the smooth inclusion ${\mathcal P}_g^\infty(M)\to C^\infty([0, 1], M)$. 
Since $DM = M$ and $D[0, 1]$ is the subspace $I$ of ${\mathbb R}$ (see \cite[Lemma 3.16]{C-S-W}), it follows that 
$\lambda : D({\mathcal P}_g^\infty(M)) \to {\mathcal P}_g(M)$ is continuous. Therefore, the composite 
$\xi':= \lambda_*\circ \xi : S^D_\bullet({\mathcal P}_g^\infty(M)) \to \text{Sing}_\bullet({\mathcal P}_g(M))$ is a morphism of simplicial sets. 

The following lemma is a key to proving Theorem \ref{thm:DR}. 

\begin{lem}\label{lem:smoothing_P} Let $M$ be a simply-connected manifold. Then, one has a sequence of quasi-isomorphisms 
\[
\xymatrix@C18pt@R15pt{
 A_{\text{PL}}(\text{\em Sing}_\bullet({\mathcal P}_g(M)))\otimes_{\Q}\R \ar[r]^{(\xi')^*}_{\simeq} &  
 A_{\text{PL}}(S^D_\bullet({\mathcal P}_g^\infty(M)))\otimes_{\Q}\R 
 \ar[r]^-{\zeta}_-{\simeq} & A_{\text{DR}}(S^D_\bullet({\mathcal P}_g^\infty(M))). 
}
\]
\end{lem}

\begin{proof} The result \cite[Corollary 3.5]{K} allows us to obtain a quasi-isomorphism $\zeta$. We consider a commutative diagram 
\[
\xymatrix@C10pt@R18pt{
H^*(A_{DR}(S^D_\bullet({\mathcal P}_g^\infty(M))) & \text{Tor}_{A_{DR}(S^D(M^{\times 2}))}^*(A_{DR}(S^D(M)), A_{DR}(S^D(M^I))) \ar[l]^-{\cong}_-{EM_1}\\
H^*(A_{PL}(S^D_\bullet({\mathcal P}_g^\infty(M))))_\R \ar[u]_{H(\zeta)}^{\cong} & 
\text{Tor}_{A_{PL}(S^D(M^{\times 2}))}^*(A_{PL}(S^D(M)), A_{PL}(S^D(M^I)))_\R \ar[l]_-{EM_2} \ar[u]_{\text{Tor}(\zeta, \zeta)}^{\cong}
\\
H^*(A_{PL}(\text{Sing}_\bullet({\mathcal P}_g(M))))_\R \ar[u]_{H((\xi')^*)} &\text{Tor}_{A_{PL}(M^{\times 2})}^*(A_{PL}(M), A_{PL}(M^I))_\R \ar[l]^-{\cong}_-{EM_3}
\ar[u]_{\text{Tor}(\xi^*, \xi^*)}
}
\]
in which the horizontal maps are induced by the Eilenberg--Moore map; see \cite[20.6]{Halperin}.  
Here, we write $A_{PL} (X)$ for $A_{PL} (\text{Sing}_\bullet(X))$ in the right-hand corner and $( \ )_\R$ denotes the tensor product $( \ )\otimes_\Q\R$. 

By assumption, the manifold $M$ is simply connected. Then, the proofs of \cite[Theorem 5.5]{K} and \cite[20.6]{Halperin} yield that $EM_1$ and $EM_3$ are isomorphisms, respectively. 
Since $\zeta$ is a quasi-isomorphism, it follows that the vertical maps in the upper square are isomorphisms and then so is $EM_2$. 
We see that $M^I$ is smooth homotopy equivalent to $M$. Therefore, Theorem \ref{thm:smoothing} implies that ${\text{Tor}(\xi^*, \xi^*)}$ is an isomorphism. 
It turns out that $(\xi')^*$ is a quasi-isomorphism. 
\end{proof}

\begin{proof}[Proof of Theorem \ref{thm:DR}]
For the translation groupoid 
$$\xymatrix@C15pt@R5pt{[G\times {\mathcal P}_G^\infty (M)  \ar@<-0.5ex>[r]_-{t} \ar@<0.5ex>[r]^-{s}& {\mathcal P}_G^\infty(M)]}$$ in which $s$ and $t$ are defined by the projection  and the action, respectively, we see that $s\circ \widetilde{p} = t\circ \widetilde{p}$. 
Then the map 
$\widetilde{p}^* : H^*_{DR}(L^\infty(M/G)) \to H^*_{DR}({\mathcal P}_G^\infty(M))$ induced by $\widetilde{p}$ factors through 
$\R\Box_{\R[G]^\vee}H^*_{DR}({\mathcal P}_G^\infty(M))$.  We show that the morphism $\widetilde{p}^*$ of algebras in the theorem is an isomorphism. Consider a commutative diagram
\[
\xymatrix@C15pt@R18pt{
H^*(C^0(S^1, M/G); \R) \ar[r]^{(\widetilde{\xi})^*}_{\cong}& H_{DR}^*(C^\infty(S^1, M/G)) \\
H^*(L(M/G); \R) \ar[r]^{(\widetilde{\xi})^*} \ar[u]^{(q^*)^*}_{\cong} \ar[d]_{\widetilde{p}^*}^{\cong} & H_{DR}^*(L^\infty(M/G)) \ar[u]_{(q^*)^*}^{\cong}\ar[d]^{\widetilde{p}^*}\\
\R\Box_{\R[G]^\vee} H^*({\mathcal P}_G(M); \R)  \ar[r]_{H(\widetilde{\xi}_1)} & \R\Box_{\R[G]^\vee} H^*_{DR}({\mathcal P}^\infty_G(M)), 
}
\]
where $\widetilde{\xi}_1$ denotes the composite of quasi-isomorphisms in Lemma \ref{lem:smoothing_P} and $\widetilde{\xi}$ is the composite of the morphism induced by $\xi$ in 
Theorem \ref{thm:smoothing} and the quasi-isomorphism 
$\zeta : A_{PL}(K)\otimes_{\Q}\R \to A_{DR}(K)$ for a simplicial set $K$ in \cite[Corollary 3.5]{K}. Thus, Theorem \ref{thm:smoothing} implies that $(\widetilde{\xi})^*$ in the top row is an isomorphism. We observe that, by Remark \ref{rem:L(M/G)}, the map $\widetilde{p}^*$ on the left-hand side is an isomorphism. 

Lemma \ref{lem:I_and_S1} implies that map $(q^*)^*$ on the right-hand side induced by the smooth map $q : I \to S^1$ in $\mathsf{Diff}$ is an isomorphism. A usual argument on the quotient map $q : I \to S^1$ shows that the projection induces a weak homotopy equivalence $(q^*) : C^0(S^1, M/G) \to L(M/G)$ and hence the left-hand side map 
$(q^*)^*$ is an isomorphism. 
Since the map $H(\widetilde{\xi}_1)$ in the lowest row is also an isomorphism, it follows that 
the right-hand side map $\widetilde{p}^*$  is an isomorphism. We have the result. 
\end{proof}

The following result follows from Theorems \ref{thm:computation_I} and \ref{thm:DR}. 

\begin{thm}\label{thm:deRhamP}  
One has sequences  
\[
\xymatrix@C5pt@R12pt{
H^*_{DR}(L^\infty \R P^{2m+1}) \ar[r]^-{\widetilde{p}^*}_-\cong  & \R\Box_{\R[G]^\vee}H^*_{DR}({\mathcal P}_G^\infty(S^{2m+1})) \\
& \big(\wedge (\alpha \circ  \mathsf{It}(v_{2m+1}))\otimes \R[\alpha \circ  \mathsf{It}([v_{2m+1})] \big)^{\oplus 2}  \ar[u]_{\cong}  \  \  \      \text{and} \\
H^*_{DR}(L^\infty \R P^{2m}) \ar[r]^-{\widetilde{p}^*}_-\cong  & \R\Box_{\R[G]^\vee}H^*_{DR}({\mathcal P}_G^\infty(S^{2m})) \\
& \big(\wedge(\alpha \circ  \mathsf{It}(v_{2m}[v_{2m}])) \otimes \R[\alpha \circ  \mathsf{It}(1[v_{2m}|v_{2m}])\big)\oplus \R \ar[u]_{\cong}
}
\] 
of isomorphisms of algebras, 
where $v_n$ denotes the volume form on $H^*_{DR}(S^n)$, $\mathsf{It}$ and $\alpha$ are Chen's iterated integral map and the factor map, 
respectively; see Appendices \ref{sect:S-deRham} and \ref{sect:ChenIt}. 
\end{thm}

\begin{proof} We prove the result on the isomorphisms in the second sequence. 
By virtue of Theorem \ref{thm:general_main}, we see that the composite 
\[
\xymatrix@C15pt@R12pt{
\Omega(M) \otimes_{1\times g} \overline{B}(\widetilde{\Omega}(M)) \ar[r]^-{\mathsf{It}} & \Omega({\mathcal P}_g^\infty(M)) \ar[r]^-{\alpha} & 
A_{DR}(S^D_\bullet({\mathcal P}_g^\infty(M))) 
}
\]
is a quasi-isomorphism for $M = S^n$; see Proposition \ref{prop:Afactor_map}. It is immediate to show that $1\times g$ is an {\it induction}. 
Then, Theorem \ref{thm:general_main} is applicable to this case. 
Moreover,  it follows from Theorem \ref{thm:DR}, Lemma \ref{lem:smoothing_P} and the computation in Theorem \ref{thm:computation_I} that 
\[
H^*_{DR}(L^\infty \R P^{2m}) \cong H^*(A_{DR}(S^D_\bullet({\mathcal P}_\tau^\infty(S^{2m})))\oplus \R \cong H^*({\mathcal P}_\tau (S^{2m})); \R)\oplus \R 
\]
as algebras. Observe that $\widetilde{p}^*$ gives the first isomorphism. 
The same computation with the cyclic bar complex as in \cite[Theorem 2.1]{K96} allows us to deduce that $v_{2m}[v_{2m}]$ 
and $1[v_{2m}|v_{2m}]$ are non-exact cocycles.  
With the indecomposable elements of the second algebra in Theorem \ref{thm:computation_I},  we see that 
$\deg (x\otimes u) = 4m-1 = \deg v_{2m}[v_{2m}]$ and $\deg w = 4m-2 =\deg 1[v_{2m}|v_{2m}]$. Thus, we have the result. The same argument as above enables us to obtain the isomorphisms in the first sequence.  
\end{proof}

\medskip
\noindent
{\it Acknowledgements.}
The author thanks Takahito Naito,  Shun Wakatsuki and Toshihiro Yamaguchi for many valuable suggestions 
on the first draft of this manuscript. 
He is also grateful to Jean-Claude Thomas and Luc Menichi for fruitful discussions on the computation in Section \ref{sect:example(s)}. The author would like to express his gratitude to Jim Stasheff and the referee for suggestions in revising a version of this manuscript. 
This work was partially supported by JSPS KAKENHI Grant Numbers JP19H05495 and JP21H00982. 


\appendix

\section{The free loop space of a global quotient}\label{sect:App2}
Let $G \to M \stackrel{p}{\to} M/G$ be a principal $G$-bundle with {\it discrete} fibre $G$. Under the same notations as in Section \ref{sect:example(s)}, we show the following 
proposition. 
\begin{prop}\label{prop:An_orbifold_stack} \text{\em (cf. \cite[Proposition 5.9]{BGNX})} The map 
$\overline{p} : EG \times_G{\mathcal P}_G(M) \to L(M/G)$ 
induced naturally by the projection $p : M \to M/G$ is a weak homotopy equivalence. 
\end{prop}

Let $G$ be a discrete group acting on a space $M$. 
Let $g$ be an element of $G$. We recall a fibre square of the form 
\[
\xymatrix@C25pt@R15pt{
{\mathcal P}_g(M) \ar[r] \ar[d]_{q_g} & M^{[0, 1]} \ar[d]^-{(ev_0, ev_1)} \\
M \ar[r]_-{\phi_g} & M\times M,
}
\]
where $ev_i$ denotes the evaluation map at $i$ for $i = 0, 1$ and  $\phi_g$ is the map defined by $\phi_g(x) = (x, gx)$ for $x \in M$. 
Thus, for each $m \in M$, we have a fibration 
\[
\xymatrix@C30pt@R15pt{
{\mathcal P}_G^m(M)\ar[r] & {\mathcal P}_G(M) \ar[r]^-{q:=\coprod q_g} &M, 
}
\]
where  ${\mathcal P}_G^m(M) = \coprod_{g\in G}{\mathcal P}_g^m(M)$ and 
$
{\mathcal P}_g^m(M) :=\{ \gamma : [0, 1] \to M  \mid \gamma(0)= m, \gamma(1) =g\gamma(0)=gm \}.
$

Since the projection $p_2 : EG \times M \to M$ is a $G$-equivariant map and a homotopy equivalence, it follows that 
$\widetilde{p_2} :  {\mathcal P}_G^m(EG\times M) \to  {\mathcal P}_G^m(M)$ induced by $p_2$ is weak homotopy equivalent and hence 
so is $\widetilde{p_2} : EG\times_G  {\mathcal P}_G(EG\times M) \to  EG\times_G{\mathcal P}_G(M)$.
Moreover, Proposition \ref{prop:An_orbifold_stack} is applicable to the $G$-bundle $G \to EG \times M \to EG\times_G M$. 

\begin{cor}\label{cor:L_Borel} \text{\em (cf. \cite[Theorem 2.3]{L-U-X})} One has weak homotopy equivalences
\[
\xymatrix@C25pt@R14pt{
EG\times_G{\mathcal P}_G(M) &  
EG\times_G  {\mathcal P}_G(EG\times M) \ar[l]_-{\widetilde{p_2}}^-\simeq \ar[r]^-{\overline{p}}_-\simeq& L(EG\times_G M). 
}
\]
\end{cor}

The proof of Proposition \ref{prop:An_orbifold_stack} we present here differs from that of [32, Theorem 2.3] which assumes $G$ to be finite.

\begin{proof}[Proof of Proposition \ref{prop:An_orbifold_stack}] With the same notations as above, we have a commutative diagram 
\[
\xymatrix@C25pt@R15pt{
{\mathcal P}_G^m(M)\ar[d] \ar[r]^{\overline{p}} &  \Omega_{[m]}(M/G) \ar[d] \\
EG\times_G {\mathcal P}_G(M) \ar[d]_-{1\times_G ev_0} \ar[r]^{\overline{p}} & L(M/G) \ar[d]^{ev_0} \\
EG\times_G M \ar[r]_{\widetilde{\pi}}& M/G 
}
\]
in which two vertical sequences are fibrations and $\widetilde{\pi}$ is induced by the projection 
$\pi : EG \times M \to M$ in the second factor. 
Observe that the result \cite[Proposition 3.2.2]{Nei} yields the left-hand side fibration; see also  \cite[Proposition B.1]{Matsuo}. 
The maps $\pi$ and $\widetilde{\pi}$ give a morphism of fibrations from $G \to EG \times M \to EG\times_G M$ to $G \to M \to M/G$ which is the identity map on the fibres. 
Since $EG$ is contractible, it follows that 
$\widetilde{\pi}$ is a weak homotopy equivalence. Therefore, in order to prove the result, it suffices to show that  the map 
$\overline{p} :   {\mathcal P}_G^m(M) \to  \Omega_{[m]}(M/G) $ is weak homotopy equivalent. 

We consider the map $\overline{p}_* : \pi_n({\mathcal P}_G^m(M), \widetilde{\gamma}) \to \pi_n(\Omega_{[m]}(M/G), \gamma)$ induced by the projection $p : M \to M/G$ 
for $n \geq 1$.  For an element $\beta :  S^n \to \Omega_{[m]}(M/G)$ in $\pi_n(\Omega_{[m]}(M/G), \gamma)$, the adjoint gives rise to a map $ad(\beta) : \Sigma S^n \to M$, where $\Sigma S^n$ denotes the unreduced suspension of $S^n$. 
Let $f_\beta : 
S^n \times I \to M$ be the composite of $ad(\beta)$ and the projection $S^n \times I \to \Sigma S^n$. 
Observe that $f_\beta( * , t ) = \gamma(t)$. 
Since $\Sigma S^n$ is homeomorphic to $S^{n+1}$, it is immediate that $(f_\beta)_* (\pi_1(S^n\times I)) = \{0\} \subset p_*(\pi_1(M))$. Then the lifting theorem for covering spaces allows us to obtain a lift $\widetilde{f_\beta}$ of $f_\beta$. 
We see that $(\overline{p})_* (ad(\widetilde{f_\beta})) = \beta$. Thus, the map $\overline{p}_*$ is surjective.  

Let $H : S^n \times I \to \Omega_{[m]}(M/G)$ be a homotopy from $\overline{p}(\alpha_0)$ to $\overline{p}(\alpha_1)$ 
based at  $\gamma$, where 
$\alpha_0$ and $\alpha_1$ are elements in  $\pi_n({\mathcal P}_G^m(M), \widetilde{\gamma})$. Then, we have a lift 
$K : S^n \times I \times I \to M$ of 
the adjoint $ad(H) : S^n \times I \times I \to \Sigma S^n \times I \to M/G$. It follows from the uniqueness of the lift that 
$ad(K) : S^n \times I \to {\mathcal P}_G^m(M)$  is a homotopy from $\alpha_0$ to $\alpha_1$ based at $\widetilde{\gamma}$. 

The same argument with the lifting theorem as above allows us to show the bijectivity of the map $\overline{p}_* : \pi_0({\mathcal P}_G^m(M)) \to \pi_0(\Omega_{[m]}(M/G))$. This completes the proof. 
\end{proof}

\section{Diffeological spaces}\label{sect:App} 
We begin by reviewing the definitions of a diffeology and a diffeological space.
A good reference for the subjects is the book \cite{IZ}. 
Additionally, we recall the definitions of Souriau--de Rham complex, the singular de Rham complex for a diffeological space and a morphism of differential graded algebras connecting the two complexes; 
see \cite{K} for the details. 

\begin{defn}
Let $X$ be a set. A set  $\D$ of functions $U \to X$ for each open subset $U$ in ${\mathbb R}^n$ and each $n \in {\mathbb N}$ 
is a {\it diffeology} of $X$ if the following three conditions hold: 
\begin{enumerate}
\item Every constant map $U \to X$ for all open subset $U \subset {\mathbb R}^n$ is in $\D$;
\item If $U \to X$ is in $\D$, then for any smooth map $V \to U$ from an open subset $V$ of ${\mathbb R}^m$, the composite 
$V \to U \to X$ is also in $\D$; 
\item
If $U = \cup_i U_i$ is an open cover and $U \to X$ is a map such that each restriction $U_i \to X$ is in $\D$, 
then the map $U \to X$ is in $\D$. 
\end{enumerate}
\end{defn}

We call an open subset of $\R^n$ a {\it domain}. 
A {\it diffeological space} $(X, \D)$ consists of a set $X$ and a diffeology $\D$ of $X$. 
An element of a diffeology $\D$ is called a {\it plot} of $X$. 
Let $(X, \D^X)$ and $(Y, \D^Y)$ be diffeological spaces. A map $X \to Y$ is {\it smooth} if for any plot $p \in \D^X$, the composite 
$f\circ p$ is in $\D^Y$. 
All diffeological spaces and smooth maps form a category $\mathsf{Diff}$. It is worthwhile mentioning that the category $\mathsf{Diff}$ is complete, cocomplete and cartesian closed. Moreover, the category of manifolds embeds into $\mathsf{Diff}$; see also \cite[Section 2]{C-S-W}.

Let $\{(X_i, \D_i) \}_{i \in I}$ be a family of diffeological spaces. Then, the product $\Pi_{i\in I}X_i$ has a diffeology $\D$, called the {\it product diffeology}, defined to be the set of all maps $p :  U \to \Pi_{i\in I}X_i$ from a domain such that $\pi_i\circ p$ are plots of $X_i$ for each $i \in I$, where 
$\pi_i :  \Pi_{i\in I}X_i \to X_i$ denotes the canonical projection. 
Moreover, for diffeological space $X$ and $Y$, the set $F:=C^\infty(X, Y)$ of smooth maps from $X$ to $Y$ is endowed with the {\it functional diffeology} $\D_F$ defined by 
$\D_F: =\{p : U \to F \mid U \ \text{is domain and} \ ad(p) : U \times X \to Y \ \text{is smooth} \}$, where $ad(p)$ denotes the adjoint to $p$.

The category $\mathsf{Diff}$ is related to $\mathsf{Top}$ the category of topological spaces with adjoint functors. Let $X$ be a topological space. Then the {\it continuous diffeology} is defined by the family of continuous maps $U \to X$ from domains. This yields a functor 
$C : \mathsf{Top} \to \mathsf{Diff}$. 
For a diffeological space $(M, \D_M)$, we say that a subset $A$ of $M$ is {\it D-open} if for every plot 
$p \in \D_M$, the inverse image $p^{-1}(A)$ is an open subset of the domain of $p$ equipped with the standard topology. The family of D-open subsets of $M$ defines a topology of $M$. Thus, by giving the topology to each diffeological space, we have a functor 
$D: \mathsf{Diff} \to \mathsf{Top}$ which is the left adjoint to $C$; see \cite{S-Y-H, C-S-W} for more details. The topology for a diffeological space $M$ is called the {\it D-topology} of $M$. 


For a finite-dimensional manifold $M$, the set of all smooth maps from domains to $M$ define a diffeology $\mathcal{D}^M$, which is called the {\it standard diffeology}.
Thus, a functor 
$\ell : \mathsf{Mfd} \to \mathsf{Diff}$ is defined by $\ell(M) = (M, \mathcal{D}_M)$, where $\mathsf{Mfd}$ is the category consisting of finite-dimensional manifolds and smooth maps.  Observe that $\ell$ is a fully faithful embedding. 
Moreover, we see that the forgetful functor $U$ from $\mathsf{Mfd}$ to 
$\mathsf{Top}$ factors through the category $\mathsf{Diff}$. 
We summarize the categories and functors mentioned above with the diagram
\[
\xymatrix@C55pt@R15pt{
\mathsf{Mfd} \ar[r]_{\ell :  \ \text{fully faithful}}  \ar@/^1.5pc/[rr]^{U :  \ \text{forgetful functor}}_{\circlearrowright}
  & \mathsf{Diff} 
\ar@<0.8ex>[r]^-{D}
& \mathsf{Top}. \ar@<0.8ex>[l]^-{C}_-{\bot}  
}
\]
 

\subsection{The Souriau--de Rham complex, the simplicial de Rham complex and the factor map}\label{sect:S-deRham}
We here recall the de Rham complex $\Omega^*(X)$ of a diffeological space $(X, \D^X)$ in the sense of 
Souriau \cite{So}. 
For an open set $U$ of ${\mathbb R}^n$, let $\D^X(U)$ be the set of plots with $U$ as the domains and 
$$\Lambda^*(U) = \{h : U \longrightarrow \wedge^*(\oplus_{i=1}^{n} {\mathbb R}dx_i ) \mid h \ \text{is smooth}\}$$
the usual de Rham complex of $U$.  
We can regard $\D^X( \ )$ and $\Lambda^*( \ )$  as functors from $\mathsf{Euc}^{\text{op}}$ to $\mathsf{Sets}$ the category of sets. 
A $p$-{\it form} is a natural transformation from $\D^X( \ )$ to $\Lambda^*( \ )$. Then, the de Rham complex $\Omega^*(X)$ is defined by the 
cochain algebra consisting of $p$-forms for $p\geq 0$; that is, $\Omega^*(X)$ is the direct sum of 
\[
\Omega^p(X) := \Set{
\xymatrix@C35pt@R8pt{
\mathsf{Euc}^{\text{op}} \rtwocell^{\D^X}_{\Lambda^p}{\hspace*{0.2cm}\omega} & 
\mathsf{Sets} }
| \omega \ \text{is a natural transformation}
}
\]
with the cochain algebra structure induced by that of $\Lambda^*(U)$ pointwisely.  In what follows, we may write $\omega_p$ for 
$\omega_U(p)$ for a plot $p : U \to X$. 
The de Rham complex defined above is a generalization of the usual de Rham complex of a manifold. 

\begin{rem}\label{rem:tautological_map} 
Let $M$ be a manifold and $\wedge^*(M)$ the usual de Rham complex of $M$. 
We recall the {\it tautological map} $\theta : \wedge^*(M) \to \Omega^*(M)$ defined by 
$
\theta(\omega) = \{ p^*\omega\}_{p \in {{\mathcal D}^{M}}},
$
where $\D^M$ denotes the standard diffeology of $M$. 
Then, it follows that $\theta$ is an isomorphism of cochain algebras; see \cite[Section 2]{H-V-C}. 
\end{rem}

Let ${\mathbb A}^{n}:=\{(x_0, ..., x_n) \in {\mathbb R}^{n+1} \mid \sum_{i=0}^n x_i = 1 \}$ 
be the affine space which is diffeomorphic to the Euclidean space ${\mathbb R}^n$. 
Let $(A_{DR}^*)_\bullet$ be the simplicial cochain algebra
defined by $(A^*_{DR})_n := \wedge^*({\mathbb A}^{n})$ for each $n\geq 0$. 
Then, for a simplicial set $K$, we define the de Rham complex $A^*_{DR}(K)$ by the set of simplicial maps from $K$ to $(A_{DR}^*)_\bullet$ endowed with 
a differential graded algebra structure induced by that of each $(A^*_{DR})_n$; see \cite[Section 2]{K} for more details. 
For a general simplicial cochain algebra $A_\bullet$,  we denote by $A(K)$ the cochain algebra  
\[
\mathsf{Sets^{\Delta^{op}}}(K, A_\bullet):= \Set{
\xymatrix@C35pt@R8pt{
\mathsf{\Delta}^{\text{op}} \rtwocell^{K}_{A_\bullet}{\hspace*{0.2cm}\omega} & 
\mathsf{Sets} }
| \omega \ \text{is a natural transformation}
}
\]
whose cochain algebra structure is induced by that of $A_\bullet$. 

For a diffeological space $X$, let $S^D_\bullet(X)$ be the simplicial set defined by 
$ 
 S^D_\bullet(X)_{\text{aff}}:= \{ \{ \sigma : {\mathbb A}^n \to X \mid \sigma \ \text{is a $C^\infty$-map} \} \}_{n\geq 0}. 
$
Thus, we have the {\it singular de Rham complex} 
$A_{DR}^*(S^D_\bul(X)_{\text{aff}})$ for a diffeological space $X$. 
This is regarded as 
a diffeological variant of Sullivan's simplicial polynomial form on a topological space.  
In fact, for a space $X$, the polynomial-de Rham complex $A_{PL}^*(X)$ is defined by $A_{PL}^*(X) := A_{PL}(\text{Sing}_\bullet(X))$ with the simplicial differential graded algebra 
$(A_{PL}^*)_\bullet$ of polynomial forms; see \cite{B-G} and \cite[II 10 (a), (b) and (c)]{FHT}.

\begin{rem}\label{rem:aff_vs_st}
By \cite[Lemma 3.1]{Kihara_19}, we have that the inclusion $i : \Delta^n_{st}  \to {\mathbb A}^n$ is smooth; see Remark \ref{rem:3-2} for  the notation. 
Moreover, the consideration at the end of 
\cite[Section 5]{K} yields that the chain map induced by  $i^* : S^D_\bullet(X)_{\text{aff}} \to S^D_\bullet(X)$ is a quasi-isomorphism for every diffeological space $X$; see also \cite[Table 1, page 959]{K}. 
The de Rham theorem holds for the singular de Rham cohomology; see \cite[Theorem 2.4 and Corollary 2.5]{K}. 
Therefore, the results on the singular de Rham cohomology $H^*(A_{DR}^*(S^D_\bul(X)_{\text{aff}}))$ in \cite{K} hold for the cohomology $H^*(A_{DR}^*(S^D_\bul(X)))$. 
\end{rem}

We recall the {\it factor map} 
$\alpha : \Omega^*(X) \to A_{DR}^*(S^D_\bullet(X)_{\text{aff}})$ of cochain algebras defined by \[\alpha(\omega)(\sigma) = \sigma^*(\omega).
\]
We refer the reader to the result \cite[Theorem 2.4]{K} for an important role of the factor map in the de Rham theorem for diffeological spaces. In particular, we have 
\begin{prop}\label{prop:Afactor_map} \text{\em  (\cite[Theorem 2.4]{K})} Suppose that $X$ is a manifold, more generally, a stratifold in the sense of Kreck \cite{Kreck}. 
Then the factor map $\alpha$ for $X$ is a quasi-isomorphism. 
\end{prop}


\subsection{Chen's iterated integral map in diffeology}\label{sect:ChenIt}

Let $N$ be a diffeological space and $\rho : \R \to I$ a cut-off function with $\rho(0)=0$ and $\rho(1)=1$. 
Then, we call a $p$-form $u$ on the diffeological space $I\times N$ an $\Omega^p(N)$-{\it valued function on} $I$ if 
for any plot $\psi : U \to N$ of $N$, the $p$-form $u_{\rho \times \psi}$ 
on $\R \times U$ is of type 
\[
\sum a_{i_1\cdots i_p}(t, \xi)d\xi_{i_1}\wedge \cdots \wedge d\xi_{i_p},
\]
where $(\xi_1, ..., \xi_n)$ denotes the coordinates of $U$ we fix. For such an $\Omega^p(N)$-valued function $u$ on $I$, an integration 
$\int_0^1u \ dt \in \Omega^p(N)$ is defined by
\[
(\int_0^1u \ dt)_\psi = \sum (\int_0^1a_{i_1\cdots i_p}(t, \xi) \ dt) d\xi_{i_1}\wedge \cdots \wedge d\xi_{i_p}. 
\]
Each $p$-form $u$ has the form $u = dt\wedge ((\partial /\partial t) \rfloor u) + u''$, where $(\partial /\partial t) \rfloor u$ and $u''$ are an 
$\Omega^{p-1}(N)$-valued function and an $\Omega^{p}(N)$-valued function on $I$, respectively. Let $F : I \times N^I \to N^I$ be the homotopy defined by 
$F(t, \gamma)(s) = \gamma(ts)$. The Poincar\'e operator $\int_F : \Omega(N^I) \to \Omega(N^I)$ associated with the homotopy $F$ is defined by 
$\int_F v = \int_0^1((\partial /\partial t) \rfloor F^*v)dt$. 
Moreover, for forms $\omega_1$, ..., $\omega_r$ on $N$, we define 
the {\it iterated integral} $\int \omega_1\cdots \omega_r$, which is an element in $\Omega^*(N^I)$, by $\int \omega_1 = \int_F\e_1^*\omega_1$ and 
\[
\int \omega_1\cdots \omega_r = \int_F\{J(\int \omega_1\cdots \omega_{r-1}) \wedge \e_1^*\omega_r\},
\]
where $\e_i$ denotes the evaluation map at $i$, $Ju =(-1)^{\deg u}u$ and $\int \omega_1\cdots \omega_r =1$ if $r=0$; see \cite[Definition 1.5.1]{C}. 
We observe that the Poincar\'e operator is of degree $-1$ and then $\int \omega_1\cdots \omega_r$ is of degree $\sum_{1\leq i \leq r}(\deg \omega_i -1)$. 

With a decomposition of the form $\widetilde{\Omega}^1(N)\oplus d\Omega^0(N)$, 
we have a cochain subalgebra $\widetilde{\Omega}(N)$ of $\Omega(N)$ which satisfies the condition that $\widetilde{\Omega}^p(N) = \Omega(N)$ for $p> 1$ and $\widetilde{\Omega}^0(N)=\R$. The cochain algebra $\widetilde{\Omega}(N)$ gives rise to the normalized bar complex 
$B(\Omega(N), \widetilde{\Omega}(N), \Omega(N))$; see \cite[\S 4.1]{C}. Consider the pullback diagram 
\begin{eqnarray}\label{eq:diagram_It}
\xymatrix@C25pt@R15pt{
E_f \ar[r]^-{\widetilde{f}} \ar[d]_{p_f} & N^I \ar[d]^{\e_0\times \e_1}\\ 
M \ar[r]_-{f} & N\times N
}
\end{eqnarray}
of $\e_0\times \e_1 : N^I \to N\times N$ along a smooth map $f : M \to N\times N$. 
We write $\overline{B}(\widetilde{\Omega}(N))$ for 
$B(\R, \widetilde{\Omega}(N), \R)$. 
Then we have a map  
\[
\mathsf{It} : \Omega(M)\otimes_{\Omega(N)\otimes\Omega(N)}B(\Omega(N), \widetilde{\Omega}(N), \Omega(N))\cong 
\Omega(M) \otimes_f \overline{B}(\widetilde{\Omega}(N)) \to \Omega(E_f)
\]
defined by 
$\mathsf{It} (v\otimes [\omega_1| \cdots | \omega_r])=
p_f^*v\wedge \widetilde{f}^*\int \omega_1\cdots \omega_r$. 
Observe that the domain of $\mathsf{It}$ gives rise to the differential on 
$\Omega(M) \otimes_f \overline{B}(\widetilde{\Omega}(N))$ by definition. Since $\rho(0)=0$ and $\rho(1)=1$ for the cut-off function $\rho$ which we use when defining the $\Omega^p(N)$-valued function on $I$, it follows that the result 
\cite[Lemma 1.4.1]{C} remains valid.  
Then the formula of iterated integrals with respect to the differential in \cite[Proposition 1.5.2]{C} implies that 
$\mathsf{It}$ is a well-defined morphism of differential graded $\Omega^*(M)$-modules.


The following theorem enables us to compute the singular de Rham cohomology of a pullback diffeological space with a bar complex via Chen's iterated integral map and the factor map mentioned above; see Remark \ref{rem:aff_vs_st}. 

\begin{thm}\label{thm:general_main} \text{\em (\cite[Theorem 5.2]{K})} Suppose that, in the pullback diagram \text{\em (\ref{eq:diagram_It})}, the diffeological space $N$ is 
simply connected and $f$ is an induction; that is, $p$ is a plot of $M$ if and only if $f\circ p$ is a plot of $N\times N$. 
Assume further that the factor maps for $N$ and $M$ 
are quasi-isomorphisms and each vector space $H^i(S^D_\bullet(N))$ is of finite dimension. Then the composite 
$\alpha \circ  \mathsf{It} : \Omega^*(M)\otimes_f \overline{B}(\widetilde{\Omega}(N)) \to 
\Omega(E_f) \to A^*_{DR}(S^D_\bullet(E_f))$ is a quasi-isomorphism of 
$\Omega^*(M)$-modules.  
\end{thm}

\section{$C^\infty(S^1, M)$ versus $L^\infty M$}\label{sect:App3}
Let $M$ be a diffeological space. We have two {\it free loop spaces} of $M$. One of them is the diffeological space $C^{\infty}(S^1, M)$ of smooth maps from the circle $S^1$ to $M$ with the functional diffeology. Another one is the diffeological space  $L^\infty M$ which fits in the pullback diagram
\begin{eqnarray}\label{eq:PB_diagram_1}
\xymatrix@C25pt@R15pt{
L^\infty M \ar[r] \ar[d] & M^{[0, 1]} \ar[d]^{\e_0\times \e_1} \\
M \ar[r]_-{\Delta} & M\times M
}
\end{eqnarray}
in the category $\mathsf{Diff}$, 
where $I :=[0, 1]$ is the diffeological subspace of $\R$ the Euclidean space 
and $\Delta$ denotes the diagonal map. 
We observe that $L^\infty M$ is diffeomorphic to the diffeological subspace 
of $M^I$ consisting of smooth maps $\gamma$ with $ \gamma(0) = \gamma(1)$.

Let $q : \R \to S^1$ be the smooth map defined by $q(t) = e^{2\pi \sqrt{-1}t}$. Then, the restriction $q : I \to S^1$ is smooth. 
In the category $\mathsf{Top}$, the continuous map $q : I \to S^1$ is regarded as a quotient map. The fact enables us to conclude that 
$q$ induces a weak homotopy equivalence $q^* : C^0(S^1, M) \stackrel{\simeq_w}{\longrightarrow} LM$ in $\mathsf{Top}$. 
We obtain a diffeological version of the equivalence.  

In order to define the weak homotopy equivalence between diffeological spaces, 
we first recall the smooth homotopy groups of a pointed diffeological space. 

We define an equivalence relation on a diffeological space $Z$ by $z\simeq w$ if there exists a smooth path $l : I \to Z$ such that $l(0) = z$ and 
$l(1) = w$. 
Let $S^n$ be the $n$-sphere endowed with sub-diffeology of the manifold $\R^{n+1}$. We use the north pole $\ast$ as a base point of $S^n$. 
For a pointed diffeological space $(X, x_0)$, let $C^\infty((S^n, \ast), (X, x_0))$ be the diffeological subspace of 
the mapping space $C^\infty(S^n, X)$ consisting of smooth maps that preserve base points.  Then, given a positive integer $n$, the {\it $n$th smooth homotopy group} 
$\pi_n^D(X, x_0)$ is defined by the set $C^\infty((S^n, x_0), (X. x_0))/\!\simeq$. Moreover, we define $\pi_0^D(X)$ by $X/\!\simeq$. 
We observe that, while the original smooth homotopy group $\pi_n(X, x_0)$ of a pointed diffeological space $(X, x_0)$ due to Iglesias-Zemmour \cite{IZ} is 
defined by using an iterated loop space of $X$, 
there is a natural bijection between the smooth homotopy set $\pi_n^D(X, x_0)$ and  $\pi_n(X, x_0)$; see \cite[Theorem 3.2]{C-W} for more details. 
 
By definition, we call a smooth map $f : X\to Y$ in $\mathsf{Diff}$ a {\it weak homotopy equivalence} if the induced maps 
$f_* : \pi_0^D(X) \to \pi_0^D(Y)$ and  $f_* : \pi_n^D(X, x_0) \to \pi_n^D(Y, f(x_0))$  for each $n$ and $x_0 \in X$ are bijective. 

\begin{lem}\label{lem:I_and_S1} The smooth map $q : I \to S^1$ mentioned above gives rise to a weak homotopy equivalence 
$q^* : C^\infty(S^1, M) \stackrel{\simeq_w}{\longrightarrow} L^\infty M$. 
\end{lem}

\begin{proof}
We prove that the maps $(q^*)_* : \pi_0(C^\infty(S^1, M)) \to \pi_0(L^\infty M)$ and $(q^*)_* : \pi_n(C^\infty(S^1, M), \gamma_0) \to \pi_n(L^\infty M, \gamma_0\circ q)$ induced by $q$ are bijective for $n \geq 1$ and each smooth loop $\gamma_0 : S^1 \to M$. To this end, we first show that the adjoint 
\begin{eqnarray*}
(1\times q)^* : \pi_0(\{\eta : N\times S^1\to M \mid \eta |_{\ast \times S^1}= \gamma_0\}) &\\
& \hspace{-5cm} \longrightarrow \pi_0(\{ \eta' : N \times I \to M \mid  \eta' |_{\ast \times I}= \gamma_0 \circ q,  
\eta' |_{N\times \{0\}} =   \eta' |_{N\times \{1\}} \})
\end{eqnarray*}
is bijective, where $N =S^n$ for $n \geq 1$. 
In what follows, we use the same notation for a homotopy class and its representative.   

We consider a function $\widetilde{\rho} : (-1-2\e,1+2\e) \to I$ with $\widetilde{\rho}(t) = 0$ for $t \in (-1-2\e, -1+2\e)\cup [0, 2\e)$ and 
$\widetilde{\rho}(t) = 1$ for $t \in(-2\e, 0)\cup (1-2\e, 1+2\e)$ for a sufficiently small positive number $\e$.  Let $(U, \varphi_U)$ and $(V, \varphi_V)$ be 
local coordinates of $S^1$ which satisfy the condition that $\varphi_U^{-1} = q$, $U':=\varphi_U(U)= (-\frac{1}{4}-\e, \frac{1}{4} + \e)$, 
$\varphi_V^{-1} = q$ and $V':=\varphi_V(V)= (\frac{1}{4}-\e, \frac{3}{4} + \e)$. 

We define a smooth path $\widetilde{\gamma_0} : S^1 \to M$ by 
$\widetilde{\gamma_0}|_U := \gamma_0\circ q\circ \widetilde{\rho} \circ \varphi_U$ and $\widetilde{\gamma_0}|_V := \gamma_0\circ q\circ \widetilde{\rho} \circ\varphi_V$. 
By using the map $\widetilde{\rho}$, we have a smooth map $\overline{\rho} : (-1-2\e,1+2\e) \to (-1-2\e,1+2\e)$ defined by 
 $\overline{\rho} (t)= \widetilde{\rho} (t) -1$ for $t \in (-1-2\e, 0)$ and $\overline{\rho} (t)= \widetilde{\rho} (t)$ for $t \in [0, 1+2\e)$. 
Since the map $\overline{\rho}$ is smooth homotopic to the identity map on $(-1-2\e,1+2\e)$ and $q \circ \widetilde{\rho} = q \circ \overline{\rho}$, 
it follows that $\gamma_0$ is smooth homotopic to $\widetilde{\gamma_0}$. Therefore, in order to prove the bijectivity of 
$(1\times q)_*$,  it suffices to show the bijectivity for $\widetilde{\gamma_0}$ instead of $\gamma_0$. 
In what follows, we write $\gamma_0$ for $\widetilde{\gamma_0}$. Let $A$ and $B$ be the domain and codomain of the map $(1\times q)_*$, respectively. 

We show the surjectivity of $(1\times q)_*$. 
Let $\eta'$ be an element in $B$. We consider a function $\rho : (-1-\e,1+\e) \to I$ with $\rho(t) = 0$ for $t \in [-1-\e, -1+\e)\cup [0, \e)$, 
$\rho(t) = 1$ for $t \in (-\e, 0)\cup [1-\e, 1+\e]$, $\rho(t) = t$ for $t \in [2\e, 1-2\e]$ and  $\rho(t) = t+1$ for $t \in [-1+2\e, -2\e]$. Observe that $\rho$ is smooth 
except for the point $0$. We define map 
$\eta : N \times S^1 \to M$ by $\eta|_{N\times U} = \eta'\circ (1\times \rho)\circ (1\times \varphi_U)$ and 
$\eta|_{N\times V} = \eta'\circ (1\times \rho)\circ (1\times \varphi_V)$. Since $\eta$ is constant in a neighborhood of zero, it follows that the map is a well-defined smooth map. Moreover, we see that $\eta |_{\ast \times U} = \eta' \circ (1\times \rho)\circ (1\times \varphi_U) = \gamma_0 \circ q \circ \rho \circ \varphi_U = \gamma_0 |_{U}$. The last equality follows from the fact that $\widetilde{\rho}\circ\rho = \widetilde{\rho}$ on $U'$. The same argument as above yields that $\eta |_{\ast \times V} =  \gamma_0 |_{V}$. Thus, we have $\eta |_{\ast \times S^1}= \gamma_0$. 

Moreover, we see that $\eta\circ (1\times q) = \eta'\circ (1\times \rho)$. Extending $\rho|_{[0,1]}$, we define a smooth map 
$\rho' : \R \to \R$ so that $\rho'|_{[1 ,\infty)}=1$ and $\rho'|_{(-\infty , 0]}=0$. Define a smooth $\widehat{\rho} : \R\times \R \to \R$ by 
$\widehat{\rho}(t, s) = (1-s)\rho'(t) + st$. Then the restriction $\widehat{\rho} : I\times I \to I$ gives rise to a smooth homotopy between 
$\rho$ and the identity on $I$. It follows that $(1\times q)^*(\eta) = \eta'\circ (1\times \rho) \sim \eta'$ with the smooth homotopy 
$\eta'\circ (1\times \widehat{\rho})$. We conclude that the map $(1\times q)^*$ is surjective. 

Let $\eta_0$ and $\eta_1$ be elements in $A$ with $(1\times q)^*(\eta_0) = (1\times q)^*(\eta_1)$. Then, there exists a smooth path from $\eta_0\circ (1\times q)$ to $\eta_1\circ (1\times q)$. Let $H : (N\times I)\times I \to M$ be the smooth homotopy which is the adjoint to the path. We define a map
$\widetilde{H} : (N\times S^1)\times I \to M$ by 
$\widetilde{H}|_{N\times W\times I} = H \circ (1\times \rho\times 1)\circ (1\times \varphi_W\times 1)$ for $W = U$ and $V$, respectively. We see that 
$\widetilde{H}$ is a well-defined smooth map that satisfies the condition that $\widetilde{H}(\ast, t,s )= \gamma_0(t)$. We define a smooth map 
$\rho'' : (-1-\e,1+\e) \to \R$ by $\rho'' |_{(-1-\e, 0]} = \rho |_{(-1-\e, 0]}-1$ and  $\rho'' |_{[0, 1+\e)} = \rho |_{[0, 1+\e)}$. Then, it follows that 
$q \circ \rho = q \circ \rho''$. Moreover, we define a smooth homotopy  $\rho_s : (-1-\e,1+\e) \times I \to I$ 
by $\rho_s (t) = (1-s)\rho''(t) + st$. Then the map  $\rho_s$ gives rise to smooth homotopies $\widetilde{H}_i \sim \eta_i$ for $i = 0$ and $1$. 
In fact, for example, we see that 
\begin{eqnarray*}
\widetilde{H}_0|_{N\times U} \!\!\!&=& \!\!\! H_0\circ (1\times \rho)\circ(1\times \varphi_U) =
\eta_0\circ (1\times q)\circ(1\times \rho)\circ(1\times \varphi_U) \\
&=&\!\!\! \eta_0\circ (1\times \pi)\circ(1\times \rho)\circ(1\times \varphi_U) = 
\eta_0\circ (1\times q)\circ(1\times \rho_0)\circ(1\times \varphi_U)  \\
&\sim&
\!\!\! \eta_0\circ (1\times q)\circ(1\times \rho_1)\circ(1\times \varphi_U) 
 =\eta_0|_{N\times U}, 
\end{eqnarray*}
where $\pi : \R \to S^1$ denotes the natural smooth extension of $q$.   
It turns out that $\eta_1 = \eta_0$ in $A$. We observe that the homotopy induced by $\rho_s$ fixes the path $\gamma_0$. 

The same argument as above enables us to prove 
that $q_* : \pi_0(C^\infty(S^1, M)) \to \pi_0(L^\infty M)$ is a bijection. 
We have the result.  
\end{proof}

\end{document}